\documentclass[a4paper]{amsart}

\usepackage{amsmath}
\usepackage{amssymb}
\usepackage{amsfonts}
\usepackage{amsopn}
\usepackage{amsthm}
\usepackage[all]{xy}

\usepackage[hypertex]{hyperref} 

\allowdisplaybreaks[1] 

\swapnumbers
\theoremstyle{plain}
\newtheorem{thm}{Theorem}[section]
\newtheorem{lem}[thm]{Lemma}
\newtheorem{cor}[thm]{Corollary}
\newtheorem{prop}[thm]{Proposition}
\newtheorem*{mthm}{Theorem}

\theoremstyle{definition}
\newtheorem{ntt}[thm]{}
\newtheorem{ex}[thm]{Example}
\newtheorem{rem}[thm]{Remark}
\newtheorem{dfn}[thm]{Definition}

\newtheorem{ass}[thm]{Assumption}

\newcommand{\lbr}{[\hspace{-1.5pt}[}  
\newcommand{\rbr}{]\hspace{-1.5pt}]}  

\newcommand{\FGLplus}[1]{+_{#1}} 
\newcommand{\fplus}{\FGLplus{F}} 
\newcommand{\fminus}{-_F} 
\newcommand{\ftimes}{\cdot_F} 
\newcommand{\Fplus}{\boxplus} 
\newcommand{\Fminus}{\boxminus} 
\newcommand{\Ftimes}{\boxdot} 

\newcommand{\FGR}[3]{#1 \lbr #2 \rbr_{#3}} 
\newcommand{\RMF}{\FGR{R}{M}{F}} 
\newcommand{\aug}{\epsilon} 
\newcommand{\augH}{\varepsilon} 

\newcommand{\RMFod}{\FGR{R}{M}{F}^{o\vee}} 

\newcommand{\DFG}[3]{{\mathcal D}_{#1}(#2)_{#3}} 
\newcommand{\DMF}{\DFG{}{M}{F}} 

\newcommand{\eDFG}[3]{\aug{\mathcal D}_{#1}(#2)_{#3}} 
\newcommand{\eDMF}{\eDFG{}{M}{F}} 

\newcommand{\mH}{{\mathcal H}}
\newcommand{\HFG}[3]{\mH_{#1}(#2)_{#3}} 
\newcommand{\HMF}{\HFG{}{M}{F}} 

\newcommand{\Gr}{{\mathcal Gr}} 
\newcommand{\Grr}[4]{{\mathcal Gr}^{#1}_{#2}(#3,#4)} 
\newcommand{\GRMF}{\Grr{*}{R}{M}{F}} 
\newcommand{\Sym}[3]{S^{#1}_{#2}(#3)} 
\newcommand{\SRM}{\Sym{*}{R}{M}} 
\newcommand{\DSym}{{\mathcal D}} 

\newcommand{\Lb}{\mathcal{L}} 
\newcommand{\Eb}{\mathcal{E}} 
\newcommand{\OO}{\mathcal{O}} 
\newcommand{\Tb}{\mathcal{T}} 

\DeclareMathOperator{\HH}{\mathsf{h}}   
\DeclareMathOperator{\CH}{\mathsf{CH}}  
\DeclareMathOperator{\KK}{\mathsf{k}}   
\DeclareMathOperator{\KT}{\mathsf{K}}   

\newcommand{\pair}[1]{<\hspace{-1.5pt}#1\hspace{-1.5pt}>} 

\newcommand{\ie}{i.e.\ }
\newcommand{\loccit}{loc. cit.}
\newcommand{\cf}{cf.\ }

\newcommand{\PP}{\mathbb{P}}  
\newcommand{\ZZ}{\mathbb{Z}}  
\newcommand{\LL}{\mathbb{L}}  
\newcommand{\TT}{\mathbf{t}}  

\newcommand{\mI}{\mathcal{I}} 
\newcommand{\mJ}{\mathcal{J}} 

\newcommand{\xra}[1]{\xrightarrow{#1}}  
\newcommand{\toby}[1]{\stackrel{#1}{\to}} 
\newcommand{\equalby}[2]{\stackrel{#2}{#1}}

\newcommand{\tor}{\mathfrak{t}} 
\newcommand{\Del}{\Delta} 
\newcommand{\gDel}{{\mathcal Gr}\Delta} 
\newcommand{\CC}{C} 
\newcommand{\gCC}{{\mathcal Gr}\CC} 
\newcommand{\Aa}{A} 
\newcommand{\AaGB}{\tilde A} 
\newcommand{\Bb}{B} 

\newcommand{\ee}{e} 

\newcommand{\ab}{\tau} 
\newcommand{\bb}{\zeta} 
\newcommand{\Zb}{\zeta} 
\newcommand{\zz}{z} 

\DeclareMathOperator{\id}{\mathrm{id}} 

\newcommand{\pt}{{\rm pt}} 

\newcommand{\cc}{\mathfrak{c}} 

\newcommand{\rev}{\mathrm{rev}} 

\newcommand{\pr}{\pi_\mH} 

\title[Oriented cohomology of complete flags]{Invariants, torsion indices and oriented cohomology of complete flags}

\author{B.~Calm\`es}
\address{Universit\'e d'Artois, Laboratoire de Math\'ematiques de Lens, France}
\email{baptiste.calmes@univ-artois.fr}

\author{V.~Petrov}
\address{Max-Plank-Institut f\"ur Mathematik, Bonn, Germany}\email{victorapetrov@googlemail.com}

\author{K.~Zainoulline}
\address{Department of Mathematics and Statistics, University of Ottawa, Canada}
\email{kirill@uottawa.ca}

\thanks{The first author was supported by EPSRC grant EP/E01786X/1. The second author is supported by RFBR~08-01-00756, 09-01-00878, 09-01-90304, 09-01-91333. The third author is supported by NSERC Discovery Grant}

\begin{document}

\maketitle

\tableofcontents


\section{Introduction}

Let $\mathsf{H}$ be an algebraic cohomology theory endowed with Chern classes $c_i$ such that for any two line bundles $\Lb_1$ and $\Lb_2$ over a variety $X$ we have 
\begin{equation}\label{fglad}
c_1(\Lb_1\otimes \Lb_2)=c_1(\Lb_1)+c_2(\Lb_2).
\end{equation}
The basic example of such a theory is the Chow group $\CH$ of algebraic cycles modulo rational equivalence. 

Let $G$ be a split semi-simple linear algebraic group over a field $k$ and let $T$ be a split maximal torus inside $G$ contained in a Borel subgroup $B$. 
Consider the variety $G/B$ of Borel subgroups of $G$ with respect to $T$.
In two classical papers \cite{Demazure73} and \cite{Demazure74} 
Demazure studied the cohomology ring $\mathsf{H}(G/B;\ZZ)$ and provided an algorithm to compute 
$\mathsf{H}(G/B;\ZZ)$ in terms of generators and relations.

The main object of his consideration was the so called characteristic map
\begin{equation}\label{charm}
\cc\colon S^*(M)\to \mathsf{H}(G/B;\ZZ),
\end{equation}
where $S^*(M)$ is
the symmetric algebra of the character group $M$ of $T$.
In \cite{Demazure73} Demazure interpreted this map from
the point of view of invariant theory of the Weyl group $W$ of $G$ by identifying its kernel with the ideal generated by non-constant invariants $S^*(M)^W$.
The cohomology ring $\mathsf{H}(G/B;\ZZ)$ was then replaced by a certain algebra constructed in terms of operators and defined in purely combinatorial terms.

In the present paper, we generalize most of the results of \cite{Demazure73} to the case of an arbitrary oriented cohomology theory $\HH$, \ie when \eqref{fglad} is replaced by
$$
c_1(\Lb_1\otimes \Lb_2)=F(c_1(\Lb_1),c_1(\Lb_2))
$$
where $F$ is the formal group law associated to $\HH$.
Such theories were extensively studied by Levine-Morel \cite{Levine07}, Panin-Smirnov \cite{Panin03}, Merkurjev \cite{Merkurjev02} and others. 
Apart from the Chow ring, other examples include algebraic $K$-theory, \'etale cohomology $H^*_{et}(-,\mu_m)$, $(m, char(k))=1$, Morava $K$-theories, connective $K$-theory, elliptic cohomology and the universal such theory: algebraic cobordism $\Omega$.

To generalize the characteristic map \eqref{charm}, we first introduce a substitute for the symmetric algebra $S^*(M)=\CH^*(BT)$.
This new combinatorial object, which we call a {\em formal group ring}, is denoted by $\RMF$, where $R=\HH(pt)$ is the coefficient ring, 
and can be viewed as a substitute 
of the cohomology ring of the classifying space $\HH(BT)$ of $T$. 
As in \cite{Demazure73}, we introduce a subalgebra $\DSym(M)_F$ of the $R$-linear endomorphisms of $\RMF$ generated by specific differential operators and by taking its $R$-dual we obtain $\HMF$, a combinatorial substitute for the cohomology ring $\HH(G/B;\ZZ)$.
The characteristic map \eqref{charm} then turns into the map
$$\cc \colon \RMF\to \HMF.$$ 
The Weyl group $W$ acts naturally on $\RMF$ 
and the main result of our paper (theorems  \ref{DR:surjchar} and \ref{DR:kercharmap_thm}) says that: 
\begin{mthm} \label{main_thm}
If the torsion index of $G$
is invertible in $R$ and $R$ has no $2$-torsion,
then the characteristic map is surjective and its kernel is generated by $W$-invariant elements in the augmentation ideal.
\end{mthm}
Demazure's methods to prove this theorem in the special case of the additive formal group law do not work in general, for the following reason: the main objects used in his proofs are operators $\Delta_w$ for every $w\in W$. They are defined first for simple reflections, and afterwards for any $w$ by decomposing it into simple reflections and composing the corresponding operators. It is then proved that the resulting composition is independent of the decomposition of $w$. 
For more general formal group laws, similar operators can still be defined for simple reflections (see Definition \ref{DO:defdelta}), but independence of the decomposition does not hold, as it was observed in \cite{Bressler92}. Geometrically, it can be translated into the fact that the cobordism class of a desingularized Schubert variety depends on the desingularization, and not only on the Schubert variety itself (see Lemma \ref{AG:ptcu0_lem}). 
We overcome this problem by working with suitable filtrations such that the associated graded structures are covered by the additive case of Demazure 
(see in particular Proposition \ref{DR:GrDel}). We therefore encourage the reader unfamiliar with \cite{Demazure73} to start by having a quick look at it before reading our sections \ref{DO} to \ref{PS}.

\medskip

As an immediate application of the developed techniques, we provide an efficient
algorithm for computing the cohomology ring $\HMF=\HH(G/B;\ZZ)$. 
To do this we generalize the Bott-Samelson approach introduced in \cite{BottSamelson} and \cite{Demazure74}.
For oriented topological theories, some algorithms were considered by Bressler-Evens in \cite{Bressler90,Bressler92} and for algebraic theories in characteristic $0$ by Hornbostel-Kiritchenko in \cite{Hornbostel09_pre}. See remark \ref{AM:algogood_rem} for a comparison.
\medskip

Note that the theorem also provides another approach to computing
the cohomology ring $\HH(G/B;\ZZ)$ 
by looking at the subring of invariants $\RMF^W$. Observe that in the classical
case when $\HH=\CH$ (or $K_0$) 
and $G$ is simply-connected it is known that $\RMF^W$ is a power series ring in basic polynomial invariants (resp. fundamental representations).
In general, the structure of $\RMF^W$ remains unknown.

\medskip

Finally, for the reader primarily interested in topology, let us mention that while all our proofs are algebraic and written in the language of algebraic geometry, the results apply as they are to topological cobordism or other complex oriented theories. Indeed, there is a canonical ring morphism
$$\Omega^*(G/B) \to {\rm MU}^*(G/B(\mathbb C))$$
(see \cite[Ex. 1.2.10]{Levine07}) that is an isomorphism because both are free modules over the Lazard ring with bases corresponding to each other (given by desingularized Schubert cells).

\medskip
The paper consists of three parts. 

In the first part, we generalize the results of \cite{Demazure73}
by introducing and studying the generalized characteristic map $\cc\colon \RMF\to\HMF$. 
In section \ref{FGR}, we introduce the formal group ring $\RMF$ and prove its main properties. In section \ref{DO}, we define the main operators $\Del$ and $\CC$ on $\RMF$. In sections \ref{DR} and \ref{IT}
we study the subalgebra $\DMF$ of operators generated by $\Del$ (resp.~$\CC$)
and multiplications. In section~\ref{IC} 
we define $\HMF$ and prove the main theorem.
In section \ref{PS} we introduce a product on $\HMF$ compatible
with the characteristic map $\cc$.

In the second part, we generalize some of the results of \cite{Demazure74}
to arbitrary oriented cohomology theories.
In section \ref{AC}, we discuss properties of oriented theories.
In sections \ref{OC} to \ref{BR}, 
we carry out the Bott-Samelson desingularization approach. 

In the last part, we apply the results of the first and the second parts to obtain
information about the ring structure in $\HH(G/B;\ZZ)$.
In section \ref{AG}, we prove that our algebraic replacement $\HMF$ is isomorphic (as a ring) to the oriented cohomology $\HH^*(G/B;\ZZ)$ and that the characteristic maps $\cc_{G/B}$ and $\cc$ correspond to each other via this isomorphism. In section \ref{MA}, we give an algebraic description of the push-forward to the point and we prove various formulas. In sections \ref{AM} and \ref{LN}, we explain an algorithm to compute the ring structure of $\HH^*(G/B;\ZZ)$ and the Landweber-Novikov operations on algebraic cobordism. Finally, in section \ref{EG}, we give multiplication tables for $\Omega^*(G/B;\ZZ)$ for groups $G$ of rank $2$. 

\medskip
\paragraph{\bf Notation}
Let $k$ denote a base field of arbitrary characteristic.
A variety over $k$ means a reduced irreducible scheme of finite type over $k$.
By $X$ and $Y$ we always mean smooth varieties over $k$.
The base point $Spec\;k$ is denoted by $\pt$.

A ring always means a commutative ring with a unit and $R$ always denotes a ring. 
A ring $R'$ is called an $R$-algebra if it comes equipped with an injective ring homomorphism $R\hookrightarrow R'$.
The letter $M$ always denotes an abelian group.

All formal group laws are assumed to be one-dimensional and commutative.
Let $F$ denote a formal group law and let $\LL$ denote the Lazard ring, \ie\ the coefficient ring of the universal formal group law $U$.

\medskip
\paragraph{\bf Acknowledgments}
Our sincere gratitude goes to Michel Demazure 
for his answers to our questions concerning the paper \cite{Demazure73}.
We would like to thank the authors of the Macaulay~2 software \cite{M2}, 
and, in particular Dan Grayson for his quick and 
accurate responses on the Macaulay~2 mailing list. 
We would also like to thank Vladimir Chernousov, Fabien Morel and Burt Totaro for their encouraging attention to our work.


\part{Invariants, torsion indices and formal group laws.}

\section{Formal group rings}\label{FGR}

Let $R$ be a ring, 
let $M$ be an abelian group and 
let $F$ be a formal group law over $R$.
In the present section we introduce and study 
the {\em formal group ring} $\RMF$.
For this, we use several auxiliary facts concerning 
topological rings and their completions which can be found in 
\cite[III, \S2]{Bourbaki85}.
The main result here is the decomposition theorem~\ref{FGR:sumdec}. 
At the end we provide some examples of computations of $\RMF$.

\begin{dfn}
Let $R$ be a ring and let $S$ be a set. 
Let $R[x_S]:=R[x_s, s\in S]$ 
denote the polynomial ring over $R$ 
with variables indexed by $S$. 
Let $\aug: R[x_S] \to R$ be the augmentation morphism 
which maps any $x_s$ to $0$. 
Consider the $\ker(\aug)$-adic topology on $R[x_S]$ given by
ideals $\ker(\aug)^i$, $i\ge 0$, which
form a fundamental system of open neighborhoods of $0$. 
Note that a polynomial is in $\ker(\aug)^i$ 
if and only if 
its valuation is at least $i$, 
hence, we have $\cap_i \ker(\aug)^i = \{0\}$ and 
the $\ker(\aug)$-adic topology is Hausdorff.

We define $R\lbr x_S \rbr$ to be the $\ker(\aug)$-adic completion 
of the polynomial ring $R [x_S]$.
\end{dfn}

\begin{rem}
By definition, an element in $R \lbr x_S \rbr$ 
can be written uniquely as a formal sum 
$$
\sum_{s_1,\ldots,s_m \in S} a_{s_1,\ldots,s_m}\cdot x_{s_1}\ldots x_{s_m},\qquad
a_{s_1,\ldots,s_m}\in R,
$$ 
where for any positive $m$ 
there is only a finite number of non-zero coefficients $a_{s_1,\ldots,s_m}$.
In particular, when $S$ is a finite set of order $n$, 
the ring $R\lbr x_S \rbr$ is 
the usual ring of power series in $n$ variables. 
\end{rem}

\begin{ntt}
Let $F$ be a formal group law 
over a ring $R$ (see \cite[p.4]{Levine07}). 
Given an integer $m\geq 0$ we use the notation 
$$
x\fplus y=F(x,y),\quad
m\ftimes x=\underbrace{x\fplus\ldots \fplus x}_{\text{$m$ times}}\;\text{ and }\;
(-m)\ftimes x= \fminus (m \ftimes x)
$$
By associativity of $F$, for any $m_1,m_2 \in \ZZ$ we have
$$
(m_1+m_2)\ftimes x =(m_1 \ftimes x)\fplus (m_2 \ftimes x).
$$ 
\end{ntt}

Now, here comes the definition playing a central role in the sequel.

\begin{dfn}\label{FGR:defin}
Let $M$ be an abelian group and 
let $F$ be a formal group law 
over a ring $R$.
Consider the ring $R\lbr x_M \rbr$ and 
let $\mJ_F$ be the closure of the ideal generated by the elements  
$x_0$ and $x_{\lambda+\mu}-(x_\lambda\fplus x_\mu)$ 
for all $\lambda,\mu \in M$. 
We define the {\em formal group ring} $\RMF$ to be the quotient
$$
\RMF=R \lbr x_M \rbr /\mJ_F.
$$
The class of $x_\lambda$ in $\RMF$ will be denoted by the same letter.
\end{dfn}

\begin{ntt}\label{FGR:keraug}
By definition, $R\lbr x_M \rbr$ is a complete Hausdorff ring 
with respect to the topology induced
by the kernel of the augmentation
$R\lbr x_M \rbr\to R$. Since $\mJ_{F}$ is clearly contained in this kernel,
the augmentation map $R\lbr x_M \rbr\to R$ factors through the quotient 
$\RMF$. 
Therefore,
$\RMF$ is a {\em complete Hausdorff} ring 
with respect to the $\mI_F$-adic topology,
where $\mI_F$ denotes the kernel of the augmentation $\RMF\to R$.
\end{ntt}

Let $f\colon R \to R'$ be a morphism of rings 
respecting the formal group laws, 
\ie sending every coefficient of $F$ 
to the corresponding coefficient of $F'$. 
Then, for every abelian group $M$, 
$f$ induces a ring homomorphism 
$f_\star\colon \RMF \to \FGR{R'}{M}{F'}$ 
sending $x_\lambda \in \RMF$ to 
$x_\lambda \in \FGR{R'}{M}{F'}$ 
for every $\lambda \in M$. 
This morphism sends $\mI_F$ to $\mI_{F'}$ and, 
hence, is continuous.

Let now $f\colon M\to M'$ be a morphism of abelian groups.
It induces a continuous ring homomorphism 
$\hat f\colon \RMF \to \FGR{R}{M'}{F}$ 
sending $x_{\lambda}$ to $x_{f(\lambda)}$.  
Moreover, if $f$ is surjective, then so is $\hat f$.

Finally, let $f\colon F \to F'$ be a morphism 
of formal group laws over $R$, 
\ie a formal power series $f \in R \lbr x \rbr$ such that 
$f(x \fplus y)=f(x) \FGLplus{F'} f(y)$. 
Then, $f$ induces a continuous ring homomorphism 
$f^\star\colon \FGR{R}{M}{F'} \to \RMF$ 
sending $x_\lambda$ to $f(x_\lambda)$.

We have therefore proved:
\begin{lem}\label{FGR:funct}
Via the above constructions $(-)_\star$, $\hat{(-)}$ and $(-)^\star$, the assignment taking $(R,M,F)$ to the topological ring $\FGR{R}{M}{F}$ is {\em covariant} with respect to ring morphisms $R\to R'$ and morphisms of abelian groups $M\to M'$,
and is {\em contravariant} with respect to morphisms of formal group laws $F\to F'$.
\end{lem}

\begin{ntt} \label{FGR:uniiso} Let $\LL$ be the Lazard ring
of coefficients of the universal formal group law $U$. 
For any formal group law $F$ over $R$, 
there is a unique ring homomorphism $\vartheta_F\colon \LL \to R$ 
sending the universal formal group law $U$ on $\LL$ to $F$. 
By universality, the morphism
$$
R \otimes_\LL \FGR{\LL}{M}{U} \to \RMF,\qquad 
r \otimes z\mapsto r \cdot (\vartheta_F)_\star(z)
$$ 
is an isomorphism.
\end{ntt}

\begin{ntt} Since $\RMF$ is an $R$-algebra, we may consider
the formal group law $F$ as an element of $\RMF\lbr x,y\rbr$.
Let $a,b\in \mI_F$. 
The specialization at $x=a$ and $y=b$ defines 
a pairing 
$$
\Fplus\colon \mI_F\times \mI_F \to \mI_F.
$$
Similarly, we define $\Fminus\colon \mI_F\to \mI_F$ 
using the inverse of $F$ and 
$\Ftimes\colon \ZZ\times \mI_F\to \mI_F$ 
by applying $\Fplus$ or $\Fminus$ iteratively. 
From the associativity, 
commutativity, inverse properties of $F$ 
and the continuity of the quotient map 
$R\lbr x_M \rbr \to \RMF$ it follows that 
$x_{\lambda+\mu}=x_{\lambda}\Fplus x_{\mu}$ and
$$
x_{-\lambda}=x_{-\lambda} \Fplus (x_{\lambda} \Fminus x_{\lambda}) 
= (x_{-\lambda} \Fplus x_\lambda) \Fminus x_{\lambda} 
= 0 \Fminus(x_\lambda) = \Fminus(x_\lambda).
$$
\end{ntt}

\begin{ntt}\label{FGR:twomaps}
A formal group law $F$ on $R$ 
induces via a ring homomorphism $R \to R'$ 
a formal group law on $R'$. 
In particular, if $M$ and $N$ are abelian groups, 
then using $R \to \RMF =R'$, 
we can define $\FGR{\RMF}{N}{F}$, 
which is naturally an $\RMF$-algebra. 
By functoriality, 
$\FGR{R}{M\oplus N}{F}$ is also an $\RMF$-algebra. 
We define two morphisms of $\RMF$-algebras
$$
\phi\colon \FGR{R}{M\oplus N}{F}\to 
\FGR{\RMF}{N}{F}\;\text{ and }\;
\psi\colon \FGR{\RMF}{N}{F} \to \FGR{R}{M\oplus N}{F}
$$
as follows:

The map 
$R [x_{M \oplus N}] \to R \lbr x_M \rbr \lbr y_N \rbr$ 
sending $x_{(\lambda,\gamma)}$ to $x_{\lambda}\fplus y_{\gamma}$ 
extends to a continuous map 
$f\colon R \lbr x_{M \oplus N} \rbr \to R \lbr x_M \rbr \lbr y_N \rbr$. 
Consider the composition
$$
R \lbr x_{M\oplus N} \rbr \toby{f} R \lbr x_M \rbr \lbr y_N \rbr 
\toby{g} \RMF \lbr y_N \rbr \toby{h} \FGR{\RMF}{N}{F}
$$ 
where $g$ is induced by $R \lbr x_M \rbr \to \RMF$ and 
$h$ is the quotient map. 
We have
$$
h \circ g \circ f(x_{(\lambda,\gamma)}\fplus x_{(\mu,\delta)})  
\equalby{=}{(1)} h \circ g \big((x_{\lambda}\fplus y_{\gamma})
\fplus(x_{\mu}\fplus y_{\delta})\big)= 
$$
$$ 
\equalby{=}{(2)} h \circ g \big((x_{\lambda}\fplus x_{\mu})
\fplus(y_{\gamma}\fplus y_{\delta})\big) 
\equalby{=}{(3)} h \big( (x_{\lambda}\Fplus x_{\mu}) 
\fplus (y_{\gamma} \fplus y_{\delta}) \big)= 
$$
$$
\equalby{=}{(4)} (x_{\lambda}\Fplus x_{\mu}) 
\fplus (y_{\gamma} \Fplus y_{\delta}) 
\equalby{=}{(5)} x_{\lambda+\mu} \fplus y_{\gamma + \delta}
\equalby{=}{(6)} h (x_{\lambda+\mu} \fplus y_{\gamma + \delta})= 
$$
$$ 
\equalby{=}{(7)} h \circ g (x_{\lambda+\mu} \fplus y_{\gamma + \delta}) 
\equalby{=}{(8)} h \circ g \circ f(x_{(\lambda+\mu,\gamma+\delta)}),
$$
where (1) and (8) hold 
by definition and continuity of $f$, 
(2) by associativity and commutativity of $F$, 
(3) and (7) by definition of $g$, 
(4) and (6) by definition of $h$ and (5) 
by definition of $\Fplus$. 
It shows that $h \circ g \circ f$ factors through 
$\FGR{R}{M\oplus N}{F}$ as
$$
h\circ g\circ f\colon R \lbr x_{M\oplus N} \rbr\to  
\FGR{R}{M\oplus N}{F} \stackrel{\phi}\to \FGR{\RMF}{N}{F},
$$ 
where $\phi$ is a morphism of $\RMF$-algebras.

In the opposite direction, we proceed as follows: 
We extend the functorial morphism 
$\RMF \to \FGR{R}{M\oplus N}{F}$ to a morphism 
$\RMF [ y_N ] \to \FGR{R}{M\oplus N}{F}$ 
by sending $y_\gamma$ to $x_{(0,\gamma)}$. 
It induces a morphism 
$\RMF \lbr y_N \rbr \to \FGR{R}{M \oplus N}{F}$ on completions, 
which factors through a continuous morphism 
$\psi\colon \FGR{\RMF}{N}{F} \to \FGR{R}{M\oplus N}{F}$.
\end{ntt}

\begin{thm} \label{FGR:sumdec}
The morphisms of $\RMF$-algebras 
$\phi$ and $\psi$ defined above 
are inverses to each other. 
In other words, we have an isomorphism
$$
\FGR{R}{M\oplus N}{F} \simeq \FGR{\RMF}{N}{F}.
$$
\end{thm}

\begin{proof}
We have
$$
\psi \circ \phi( x_{(\lambda,\gamma)}) 
= \psi ( x_{\lambda} \fplus y_{\gamma}) 
= \psi (x_\lambda) \fplus \psi(y_\gamma) 
= x_{(\lambda,0)} \fplus x_{(0,\gamma)} 
= x_{(\lambda,\gamma)}
$$
where the second equality holds by continuity of $\psi$.
The other composition can be checked on $y_{\gamma}$, 
since we are dealing with morphisms of $\RMF$-algebras, and 
we have $\phi \circ \psi(y_\gamma)=\phi(x_{(0,\gamma)}) = 
x_0 \fplus y_\gamma = 0 \fplus y_\gamma =y_\gamma$. 
\end{proof}

\begin{lem} \label{FGR:freeZ}
Let $M = \ZZ$. 
Then sending $x_m$ to $m \ftimes x$ defines a ring isomorphism 
$$
\FGR{R}{\ZZ}{F}\simeq R\lbr x\rbr.
$$
In particular, 
if $M$ is a free abelian group of rank one, 
then $\RMF$ is isomorphic to $R\lbr x \rbr$.
\end{lem}

\begin{proof}
The morphism $\phi:R[x_\ZZ] \to R\lbr x \rbr$ 
sending $x_n$ to $n\ftimes x$ 
extends to a morphism 
$\hat{\phi}:R \lbr x_\ZZ \rbr \to R \lbr x \rbr$. 
By continuity, 
$\hat{\phi}$ satisfies 
$\hat{\phi}(x_m \fplus x_n) = \hat{\phi}(x_m)\fplus \hat{\phi}(x_n)$ 
as well as 
$\hat{\phi}(\fminus(x_m))=\fminus\hat{\phi}(x_m)$. 
Thus $\hat{\phi}(x_{m+n})=\hat{\phi}(x_m \fplus x_n)$ and 
$\hat{\phi}$ factors through 
a $\bar{\phi}:\FGR{R}{\ZZ}{F} \to R \lbr x \rbr$. 
A map $\hat{\psi}$ in the opposite direction 
is defined similarly using completeness, 
by sending $x$ to $x_1$. 
By continuity, checking that $\bar\phi\circ\hat\psi=\id$ and
$\hat\psi\circ\bar\phi=\id$ can be done on generators. 
Namely, 
$\bar{\phi} \circ \hat{\psi}(x)=\bar{\phi}(x_1)=x$ by definition, 
and
$\hat{\psi} \circ \bar{\phi}(x_n)=\hat{\psi}(n\ftimes x)=
n \Ftimes \hat{\psi}(x) =n \Ftimes x_1 = x_n$.  
\end{proof}

\begin{cor} \label{FGR:freeM}
Let $\phi: M \to \ZZ^{\oplus n}$ be an isomorphism. 
Then $\phi$ induces an isomorphism 
$$
\RMF \simeq R \lbr x_1,\ldots,x_n \rbr.
$$
\end{cor}

\begin{proof}
It follows from Lemma~\ref{FGR:freeZ} and 
Theorem~\ref{FGR:sumdec} by induction on $n$.
\end{proof}

\begin{rem} \label{FGR:intdom}
Note that the right hand side is independent of $F$,
although the isomorphism depends on $F$. 
Also note that if $R$ is an integral domain, 
so is $\FGR{R}{\ZZ^{\oplus n}}{F}$.
\end{rem}

\begin{ex}
Similarly, one can prove that 
$\FGR{R}{\ZZ/n}{F} \simeq R\lbr x \rbr /(n\ftimes x)$ 
by sending $x_m$ to $m \ftimes x$. 
Observe that $n\ftimes x=nx+x^2q$.
In particular, 
if $n$ is invertible in $R$, then
$(n\ftimes x)=(x)$ and $\FGR{R}{\ZZ/n}{F} \simeq R$.
\end{ex}

Let us now examine what happens at a finite level in $\RMF$.

\begin{ntt}
Let $R[M]_F$ denote the subring of $\RMF$ 
defined as the image of the subring $R[x_M]$ 
by means of the composition 
$R[x_M]\to R\lbr x_M \rbr \to \RMF$. 
Then the ring $\RMF$ is the completion of $R[M]_F$ 
at the ideal $\mI_F \cap R[M]_F$.
By the functoriality of $\RMF$ (see \ref{FGR:funct}) 
the assignment
$$
(R,M,F)\mapsto R[M]_F
$$
is a {\em covariant} functor 
with respect to morphism of rings $R\to R'$
and morphisms of abelian groups $M\to M'$. 
Moreover, 
if $M \to M'$ is surjective, then so is $R[M]_F \to R[M']_F$.
\end{ntt}

\begin{ex}
If $M=\ZZ$, 
the isomorphism of Lemma~\ref{FGR:freeZ} 
maps $R[\ZZ]_F$ to polynomials in $n\ftimes x$. 
In particular, 
unless $F$ and the formal inverse 
have a finite number of nonzero coefficients only, 
$R[\ZZ]_F$ does {\em not} map to $R[x]$.
\end{ex}

\begin{rem}
The morphism $\phi$ from \ref{FGR:twomaps} 
doesn't necessarily send $R[M \oplus N]_F$ to $R[M]_F[N]_F$. 
The morphism $\psi$ from \ref{FGR:twomaps}
sends $R[M]_F[N]_F$ into $R[M \oplus N]_F$ 
but is not necessarily surjective. 
\end{rem}

\begin{ex}[\cf \ref{AC:chow}]\label{FGR:addfgl} 
The additive formal group law over $R$ is given by $F(x,y)=x+y$.
In this case we have ring isomorphisms
$$
\FGR{R}{M}{F} \simeq \prod_{i=0}^{\infty} S^i_R(M)\;\text{ and }\;
R [M]_F \simeq \bigoplus_{i=0}^{\infty} S^i_R(M)
$$
where $S^i_R(M)$ is the $i$-th symmetric power of $M$ over $R$ and 
the isomorphisms are induced 
by sending $x_\lambda$ to $\lambda\in S^1_R(M)$. 
\end{ex}

\begin{ex}[\cf \ref{AC:K0}]
The multiplicative periodic formal group law over $R$
is given by $F(x,y)=x+y-\beta xy$, 
where $\beta$ is an invertible element in $R$.
Consider the group ring 
$$
R[M]:=\left\{\sum_j r_j e^{\lambda_j}\mid r_j\in R,\; \lambda_j\in M\right\}.
$$ 
Let 
${\rm tr}: R[M] \to R$ be the trace map, \ie 
a $R$-linear map sending any $e^\lambda$ to $1$. 
Let $R[M]^\wedge$ be the completion of $R[M]$ at $\ker({\rm tr})$.
Then we have ring isomorphisms
$$
\FGR{R}{M}{F} \simeq R[M]^\wedge\;\text{ and }\;R[M]_F \simeq R[M]
$$ 
induced by
$x_\lambda \mapsto \beta^{-1}(1- e^{\lambda})$ and 
$e^\lambda\mapsto (1-\beta x_{\lambda})=(1-\beta x_{-\lambda})^{-1}$.
\end{ex}

\begin{ex}[\cf \ref{AC:connK}]
The multiplicative non-periodic formal group law over $R$
is given by
$F(x,y)=x+y- vxy$, where $v$ is not invertible in $R$.
Specializing $\RMF$ at $v=0$ or $v=\beta$ where $\beta$ is invertible in $R$, 
we obtain the formal group rings of
the previous examples.
\end{ex}

\begin{rem} 
Let $\HH$ be an oriented cohomology theory as defined in 
\cite[Def.~1.1.2]{Levine07} and
let $F$ be the associated formal group law (see Section~\ref{AC}). 
Let $M$ be the group
of characters of a split torus 
$T$ over $k$. 
Then the formal group ring $\RMF$
can be viewed as an algebraic substitute 
of the completed equivariant cohomology ring $\HH_T(\pt)^\wedge$ 
or the cohomology ring of the classifying space $\HH(BT)$.
Its finite counterpart $R[M]_F$ could then play the role of the equivariant cohomology ring $\HH_T(\pt)$ itself.
\end{rem}


\section{Differential operators $\Del_\alpha$ and
$\CC_\alpha$}\label{DO}

In the present section we introduce
two linear operators on $\RMF$.
The first operator $\Delta^F_\alpha$ 
is a generalized version of the operator 
$\Delta_\alpha$ from \cite[\S 3 and \S 4]{Demazure73}. 
Indeed, one recovers the results of \loccit{} 
when $R=\ZZ$ and $F$ is the additive formal group law
from Example~\ref{FGR:addfgl}.  
A version of the second operator $\CC_\alpha^F$ was already used in \cite[\S5]{Bressler92} for topological complex cobordism.

\begin{ntt}
Consider a reduced root system as in \cite[\S 1]{Demazure73}, 
\ie a free $\ZZ$-module $M$ of finite rank, 
a finite subset of $M$ 
whose elements are called {\em roots} and 
a map associating a coroot $\alpha^\vee \in M^\vee$ 
to a root $\alpha$, satisfying certain axioms.\footnote{This is often called a root datum instead of a root system in the literature, but we follow \cite{Demazure73}.}
The reflection 
map $\lambda \mapsto \lambda - \alpha^\vee (\lambda)\alpha$ 
is denoted by $s_\alpha$. 
\end{ntt}

\begin{ntt}
The Weyl group $W$ associated to a reduced root system 
is the subgroup of linear automorphisms of $M$
generated by reflections $s_\alpha$. 
It acts linearly on $M$ and thus by $R$-algebra automorphisms on $\RMF$ using the functoriality in $M$ of $\RMF$.
We have the following obvious equalities for any $u \in \RMF$:
$$
s_{\alpha}(u) = s_{-\alpha}(u) \qquad \text{and} \qquad s_\alpha^2 (u)=u.
$$
\end{ntt}

\begin{lem}
For any $n \in \ZZ$, 
the element $x-x \fplus (n \ftimes y)$ 
is uniquely divisible by $y$ in $R \lbr x,y \rbr$. 
\end{lem}

\begin{proof}
Since $y$ is a regular element in $R \lbr x,y \rbr$, 
we just need to prove divisibility. 
Note that for any power series $g(x,y)$, 
the series $g(x,0)-g(x,y)$ is divisible by $y$. 
Apply it to $g(x,y)=x \fplus (n \ftimes y)$.
\end{proof}

\begin{cor}\label{DO:divis}
Assume that $R$ is integral. 
For any $u \in \RMF$, 
the element $u-s_\alpha (u)$ 
is uniquely divisible by $x_{\alpha}$.
\end{cor}

\begin{proof}
First note that $\RMF$ is an integral domain by \ref{FGR:intdom}, 
so we just need to prove divisibility. 
Since 
$$
s_\alpha(x_\lambda)=x_{\lambda-\alpha^\vee(\lambda)\alpha} 
= x_\lambda \Fplus (-\alpha^\vee(\lambda) \Ftimes x_{\alpha})
$$
it holds by the previous lemma when $u=x_\lambda$. 
Then, by the formula 
\begin{equation} \label{DO:sandprod_eq}
uv -s_{\alpha}(u v) = 
(u-s_\alpha(u))v + u(v-s_\alpha(v))- (u-s_\alpha (u)) (v-s_\alpha(v))
\end{equation}
the result holds by induction on 
the degree of monomials for any element in $R[M]_F$. 
Finally, it holds by density on the whole $\RMF$.
\end{proof}

For a given root $\alpha$
we define a linear operator $\Delta_\alpha^F$ on $\RMF$ 
as follows:

\begin{dfn}\label{DO:defdelta}
First, using Corollary~\ref{DO:divis} 
we define a linear operator $\Delta_\alpha^{U}$
on $\FGR{\LL}{M}{U}$,
where $U$ is the universal formal group law 
over the Lazard ring $\LL$, as
\begin{equation*}
\Del^{U}_{\alpha}(u)= 
\frac{u-s_\alpha(u)}{x_{\alpha}},\;\text{ where }
u \in \FGR{\LL}{M}{U}.
\end{equation*}
Finally, identifying $\RMF$ with $R \otimes_\LL \FGR{\LL}{M}{U}$ 
via the isomorphism from \ref{FGR:uniiso} 
we define the desired operator $\Delta^{F}_\alpha$ on $\RMF$ as 
$\id_R \otimes \Del^{U}_{\alpha}$.
We will simply write $\Delta_\alpha(u)$ when the formal group law $F$ over $R$ is understood.
\end{dfn}

\begin{rem}
Observe that if $R$ is integral, 
then the operator $\Del^{F}$
can be defined directly 
using the same formula as for $\Del^{U}$.
If $R$ is torsion, 
then it is not the case.
Indeed, take $R=\ZZ/2$, $M=\ZZ$, 
$F$ to be an additive law 
and $\alpha=2$ to be a root. 
Then $x_\alpha=x_{1+1}=x_1+x_1=2x_1=0$ in $\RMF$.
\end{rem}

\begin{rem} \label{DO:otherconv}
We could instead have defined 
$\Del_{\alpha}$ as $\Del_{\alpha}(u)=\frac{s_\alpha (u)-u}{x_{-\alpha}}$. 
This exchanges
$\Del_\alpha$ with $-\Del_{-\alpha}$ 
so it is easy to switch from one convention to the other. Both conventions give the same classical operator when the formal group law is additive. 
\end{rem}

\begin{prop} \label{DO:eqDel}
The following formulas hold for any $u,v \in \RMF$, 
$\lambda \in M$ and $w \in W$ 
(Compare to \cite[\S 3]{Demazure73}\footnote{\label{mistake_fn}
There is a sign mistake in Equation (3) of \cite[\S 3]{Demazure73}. 
One should read $\Delta_{\alpha}=-\Delta_{-\alpha}$}).
\begin{enumerate}
\item \label{DO:defDel_item}
$\Del_{\alpha}(1)=0,\quad \Del_{\alpha}(u) x_{\alpha} = u - s_\alpha(u)$,

\item \label{DO:Delsquare_item}
$\Del_\alpha^2(u) x_{\alpha} = \Del_{\alpha}(u) + \Del_{-\alpha}(u),\quad
\Del_{\alpha}(u)x_{\alpha} = \Del_{-\alpha}(u) x_{-\alpha}$,

\item \label{DO:Delands_item}
$s_{\alpha} \Del_{\alpha}(u) = - \Del_{-\alpha}(u),\quad
\Del_{\alpha} s_\alpha(u)=-\Del_{\alpha}(u)$,

\item \label{DO:DelDeriv_item}
$\Del_{\alpha}(uv) = 
\Del_{\alpha}(u)v + u\Del_{\alpha}(v) - 
\Del_{\alpha}(u) \Del_{\alpha}(v)x_{\alpha} 
=\Del_{\alpha}(u)v + s_\alpha(u)\Del_{\alpha}(v)$,

\item \label{DO:Delw_item}
$w \Del_{\alpha} w^{-1}(u) = \Del_{w(\alpha)}(u)$. 
\end{enumerate}
\end{prop}

\begin{proof}
All formulas can be proved in $\FGR{\LL}{M}{U}$ 
(which is an integral domain) and 
then specialized to any other $\RMF$.
Formula~\eqref{DO:DelDeriv_item} follows from equation \eqref{DO:sandprod_eq} above. 
Formula~\eqref{DO:Delw_item} follows from the fact 
that $\alpha^\vee (w^{-1}(\lambda))= w(\alpha)^\vee(\lambda)$. 
All other formulas follow by definition.
\end{proof}

From now on, we fix a basis of simple roots $(\alpha_1,\ldots,\alpha_n)$, with associated simple reflections $s_1,\ldots,s_n$. This defines a length function on $W$. Let $I$ denote a sequence $(i_1,\ldots,i_l)$ of $l$ integers in $[1,n]$ and let $w(I)=s_{i_1}\cdots s_{i_l}$ be the corresponding product of simple reflections. The decomposition $I$ is {\em reduced} if $w$ has length $l$. We define the linear operator $\Del_I$ as the composition
$$\Del_I=\Del_{\alpha_{i_1}}\circ\cdots \circ\Del_{\alpha_{i_l}}.$$

The operators $\Del_I$ have the following important property:

\begin{thm}\label{DO:indep_dec}
Let $F$ be a formal group law of the form 
$F(x,y)=x+y-v\cdot xy$ for some $v \in R$. 
Let $I$ and $I'$ be two reduced decompositions of $w$ in simple reflections. 
Then $\Del_I= \Del_{I'}$.
\end{thm}
\begin{proof}
See \cite[Theorem 2 p. 86]{Demazure74}. 
The proof assumes $v=1$ but it works for any other value.
\end{proof}

When the formal group law is of the above type, the previous theorem justifies the notation $\Del_w$ instead of $\Del_I$, but in general, we have to keep the dependence on the decomposition.

\begin{rem}
Theorem~\ref{DO:indep_dec} was proved in \cite{Bressler90}
(see \loccit{ Thm.3.7}) 
for topological oriented theories. 
Moreover, the result of \cite{Bressler90} says that 
the formal group law has to be of the above form 
for such an independence of the decomposition to hold.
\end{rem}

For a given root $\alpha$ we define another linear operator
$\CC_\alpha^F$ on $\RMF$ as follows:

\begin{dfn}\label{DO:defC} 
Let $g(x,y)$ be the power series defined by
$$
x \fplus y = x + y - xy\cdot g(x,y)
$$ 
and let $\ee_\alpha$ be the element 
$g(x_\alpha,x_{-\alpha})$ in $\RMF$. 
We set
$$
\CC_{\alpha}^{F}(u)=u \ee_\alpha - \Del_\alpha^{F}(u),
\;\text{ where }u\in \RMF.
$$
We will simply write $\CC_\alpha(u)$ when 
the formal group law $F$ over $R$ is understood. For a sequence of integers $I$ in $[1,n]$, we define the operator $\CC_I$ in the same way as $\Del_I$. 
\end{dfn}

\begin{prop}\label{DO:eqCC}
The following formulas hold for any $u,v \in \RMF$, 
$\lambda \in M$ and $w \in W$. 
\begin{enumerate}
\item
$\CC_{\alpha}(1)=\ee_{\alpha},\quad
\CC_{\alpha}(x_{-\alpha})=2$,

\item
$\CC_{\alpha}(u) x_{\alpha} x_{-\alpha} = u x_{\alpha} + s_\alpha(u) x_{-\alpha},\quad 
\CC_{\alpha}(u x_{-\alpha}) = u + s_{\alpha}(u)$,

\item \label{DO:eqCCCs}
$\CC_{\alpha} s_\alpha(u)=\CC_{-\alpha}(u),\quad
s_{\alpha} \CC_{\alpha}(u) = \CC_{\alpha}(u)$,

\item \label{DO:prodCC_item}
$\CC_{\alpha}(uv) = \CC_{\alpha}(u)v + 
s_\alpha (u)\CC_{\alpha}(v) - s_{\alpha}(u) v \ee_{\alpha}
=\CC_\alpha(u)v-s_\alpha(u)\Del_\alpha(v)$,

\item
$w \CC_{\alpha} w^{-1}(u) = \CC_{w(\alpha)}(u)$,

\item
$\CC_{\alpha} \Del_{\alpha} = \Del_{\alpha} \CC_{\alpha} = 
\Del_{\alpha} \CC_{-\alpha}= 0$.
\end{enumerate}
\end{prop}
\begin{proof}
All these formulas can easily be derived 
from the ones of Proposition \ref{DO:eqDel} 
using the definition of $C_\alpha$
and the fact that 
$\ee_\alpha x_\alpha x_{-\alpha}= x_\alpha + x_{-\alpha}$.
\end{proof}

\begin{prop} \label{DO:DelWlin}
Let $W_{\alpha}$ be the subgroup of order $2$ 
generated by $s_{\alpha}$ in $W$.
Let $\RMF^{W_\alpha}$ denote the subring of fixed elements of $\RMF$ 
under the action of $W$.

If the ring $R$ has no $2$-torsion, 
then the operators $\Del_\alpha$ and $\CC_\alpha$
are $\RMF^{W_\alpha}$-linear. 
In particular, they are $\RMF^W$-linear. 
\end{prop}
\begin{proof}
For the operator $\Del_\alpha$,
by formula~\eqref{DO:DelDeriv_item} of \ref{DO:eqDel} it suffices to show that 
$\Del_{\alpha}(u)=0$ for $u \in \RMF^{W_{\alpha}}$.
According to formula~\eqref{DO:Delands_item} of \ref{DO:eqDel}, 
it is equivalent to
$\Del_{\alpha}(u)=-\Del_{\alpha}(u)$ which holds
if $\RMF$ has no $2$-torsion.

The same facts for $\CC_\alpha$ 
then follow from its definition in terms of $\Del_{\alpha}$.
\end{proof}

\begin{lem} \label{DO:Delalpha0_lem}
Assume $R$ has no $2$-torsion. Then for any root $\alpha$, the element $x_\alpha$ is regular in $\RMF$ and the set of elements $u$ such that $\Del_{\alpha}(u)=0$ is $\RMF^{W_\alpha}$.
\end{lem}
\begin{proof}
Let $\omega_1,\ldots,\omega_n$ be the fundamental weights corresponding to a choice of simple roots $\alpha=\alpha_1,\ldots,\alpha_n$ of the root system. Possibly extending $M$, we can assume that the $\omega_i$ are a basis of $M$. 
By the isomorphism $\RMF \simeq R\lbr x_{\omega_1},\ldots,x_{\omega_n}\rbr$ of Corollary \ref{FGR:freeM}, $x_\alpha$ becomes $\sum_i \pair{\alpha,\alpha_i}x_{\omega_i}+u$, $u\in \mI^2$. One easily checks that this element is regular if at least one of the $\pair{\alpha,\alpha_i}$ is regular in $R$. In particular, this is the case if $\pair{\alpha_1,\alpha_1}=2$ is regular in $R$. By equation \eqref{DO:defDel_item} of Proposition \ref{DO:eqDel}, the last part of the claim is then clear. 
\end{proof}

\begin{prop} \label{DO:DelCCfunct_prop}
Let $f:R \to R'$ be a ring morphism sending a formal group law $F$ over $R$ to a formal group law $F'$ over $R'$. Then, the operators $\Del$ and $\CC$ satisfy
$$f_* \Del_\alpha^F= \Del_\alpha^{F'} f_* \qquad\text{and}\qquad f_* \CC_\alpha^F=\CC_\alpha^{F'}f_*$$ 
where $f_*: \RMF \to \FGR{R'}{M}{F'}$ is the morphism of Lemma~\ref{FGR:funct}.
\end{prop}
\begin{proof}
This is clear by construction.
\end{proof}


\section{Endomorphisms of a formal group ring}\label{DR}

In the present section we introduce and study the subalgebra $\DMF$
of $R$-linear endomorphisms of a formal group ring $\RMF$
generated by the $\Delta$ operators of the previous section.
The main result (Theorem~\ref{DR:basisDMF}) says that operators $\Delta_{I_w}^F$,
where $w$ runs through all elements of the Weyl group and $I_w$ is any chosen 
reduced decomposition of $w$, form a basis of $\DMF$ as an $\RMF$-module. 

\begin{ntt}
Let $\RMF$ be a formal group ring corresponding to a formal group law $F$
over a ring $R$ and an abelian group $M$ (see Def.~\ref{FGR:defin}).
Let $\mI_F$ be the kernel of the augmentation map $\RMF \to R$ 
as in \ref{FGR:keraug}. 
By convention, we set $\mI^i_F=\RMF$ for any $i\leq 0$.
We define the associated graded ring 
$$
\GRMF=\bigoplus_{i=0}^{\infty} \mI_F^i/\mI_F^{i+1}.
$$
\end{ntt}

\begin{lem} \label{DR:SRM=GRMF}
The morphism of graded $R$-algebras
$$
\phi:\SRM \to \GRMF
$$ 
defined by sending $\lambda$ to $x_\lambda$ 
is an isomorphism.
\end{lem}

\begin{proof}
The map is well-defined since 
$x_{\lambda+\mu}-(x_\lambda+x_\mu)$ is in $\mI_F^2$, and it is obviously surjective. 
Let us define a map in the other direction. 
Recall that by definition $\RMF=R \lbr x_M \rbr/\mJ_F$. 
Let $\tilde\mI$ denote the kernel
of the augmentation map 
$R \lbr x_M \rbr \to R$  (see \ref{FGR:keraug}). 
Then we have 
$$
\mI_F = \tilde\mI/\mJ_F\text{ and } 
\mI^i_F/\mI^{i+1}_F=\tilde\mI^i/(\mJ_F \cap \tilde\mI^i + 
\tilde\mI^{i+1}).
$$ 
Since $\tilde\mI^i/\tilde\mI^{i+1}$ 
is given by monomials of degree $i$, 
we may define a map of $R$-modules  
$$
\psi_i\colon \tilde\mI^i/\tilde\mI^{i+1}\to S^i_R(M)
$$ 
by sending a monomial of degree $i$ in some $x_\lambda$-s 
to the symmetric product of the $\lambda$-s involved. 
This map passes to the quotient since 
$(\mJ_F \cap \tilde{\mI}^i + \tilde{\mI}^{i+1})/\tilde{\mI}^{i+1}$ 
is generated as an $R$-module by elements of the form 
$\rho.(x_\lambda+x_\mu-x_{\lambda+\mu})$ where $\rho$ 
is a monomial of degree $i-1$. 
The sum $\oplus_i \psi_i$ passes through the quotient as well
and, hence, gives the desired inverse map.
\end{proof}

\begin{ntt}
Consider now a reduced root system $M$ as in Section~\ref{DO} and
let $\RMF$ be the associated formal group ring.
Consider the operators $\Del_\alpha^F$ introduced in Def.~\ref{DO:defdelta}.
By their very definition, the operators $\Del_\alpha^F$ 
send $\mI^i_F$ to $\mI^{i-1}_F$. 
Hence, they induce $R$-linear operators of degree $-1$ 
on the graded ring $\GRMF$, denoted by $\gDel_\alpha^F$.
Similarly, we define graded versions of operators $\CC_\alpha^F$ 
(see Def.~\ref{DO:defC}) on $\GRMF$, denoted by $\gCC_\alpha^F$. 
\end{ntt}

\begin{prop} \label{DR:GrDel}
The isomorphism $\phi$ of Lemma~\ref{DR:SRM=GRMF} 
exchanges the operator $\gDel_\alpha^F$ (resp. $\gCC_\alpha^F$) 
on $\GRMF$ with $\Del_\alpha^{F_a}$ (resp. $-\Del_\alpha^{F_a}$) 
on the symmetric algebra $\SRM$, where
$F_a$ denotes the additive formal group law as in Example~\ref{FGR:addfgl} and the
$\Del_\alpha^{F_a}=-\CC_\alpha^{F_a}$ are the classical operators
of \cite{Demazure73}.
\end{prop}

\begin{proof}
Induction on the degree using \eqref{DO:DelDeriv_item} of Prop.~\ref{DO:eqDel}.
\end{proof}

\begin{dfn}\label{DR:DMFfil}
Let $\DMF$ be the subalgebra of $R$-linear endomorphisms of $\RMF$ 
generated by the $\Del_\alpha^F$ for all roots $\alpha$ and 
by multiplications by elements of $\RMF$. 
Note that by formula \eqref{DO:defDel_item} 
of Proposition \ref{DO:eqDel}, 
$\DMF$ contains $s_\alpha$ and it contains $C_\alpha^F$ by its definition. 
Let $\DMF^{(i)}$ be the sub $\RMF$-module of $\DMF$ 
generated by the $u \Del_{\alpha_1}^F\cdots \Del_{\alpha_n}^F$, 
where $u \in \mI^m_F$ and $m-n \geq i$. 
This defines a filtration on $\DMF$
$$
\DMF \supseteq \cdots \supseteq \DMF^{(i)} 
\supseteq \DMF^{(i+1)} \supseteq \cdots \supseteq 0
$$
with the property that $\DMF=\cup_i \DMF^{(i)}$.
We define the associated graded $\RMF$-module
$$
\Gr^*\DMF=\bigoplus_{i=-\infty}^\infty \DMF^{(i)}/\DMF^{(i+1)}.
$$
\end{dfn}

\begin{prop} \label{DR:DMFFil}
The filtration on $\DMF$ above has the following properties:
\begin{enumerate}
\item \label{DR:Delup_item} 
For any root $\alpha$, we have 
$$
\Del_{\alpha}^F\DMF^{(i)} \subseteq \DMF^{(i-1)},\qquad
\DMF^{(i)}\Del_{\alpha}^F \subseteq \DMF^{(i-1)}$$ 
and similarly for $\CC_\alpha^F$
\item \label{DR:Idown_item} 
For any integer $j$, we have $\mI^j_F \DMF^{(i)} \subseteq \DMF^{(i+j)}$
\item \label{DR:Dfiltrpres_item} 
For any operator $D \in \DMF^{(i)}$, we have $D(\mI^j_F)\subseteq \mI^{i+j}_F$
\item \label{DR:DMFHausdorff_item} 
$\cap_i \DMF^{(i)}=\{0\}$
\end{enumerate}
\end{prop}

\begin{proof}
The claim \eqref{DR:Delup_item} 
follows from equation \eqref{DO:DelDeriv_item} in 
Prop.~\ref{DO:eqDel} 
(resp.~equation \eqref{DO:prodCC_item} in Prop.~\ref{DO:eqCC}). 
The claim \eqref{DR:Idown_item} is obvious. 
The claim \eqref{DR:Dfiltrpres_item} follows from the fact that 
$\Del_{\alpha}^F(\mI^i_F)\subseteq \mI^{i-1}_F$. 
To prove \eqref{DR:DMFHausdorff_item} observe that any 
$D$ in $\cap_i \DMF^{(i)}$ sends $\RMF$ to $\cap_i \mI^i_F=0$ 
by \eqref{DR:Dfiltrpres_item} so is the zero operator. 
\end{proof}

\begin{ntt}
By claim (3) of Prop.~\ref{DR:DMFFil} the graded module 
$\Gr^*\DMF$ acts by graded endomorphisms on the graded ring $\GRMF$, 
which is isomorphic to $\SRM$ by Lemma~\ref{DR:SRM=GRMF}. 
Let $F_a$ denote the additive formal group law.
Then the associated subalgebra
$\DSym=\DFG{}{M}{F_a}$ coincides with one 
considered in \cite[\S 3]{Demazure73}. 
By Prop.~\ref{DR:GrDel}, 
the class of $\Del_\alpha^F$ (resp. $\CC_{\alpha}^F$) 
in $\Gr^{(-1)}\DMF$ acts by $\Del_\alpha^{F_a}$ 
(resp. $-\Del_{\alpha}^{F_a}$) on $\SRM$ and 
the class of the multiplication by $x_\alpha$ in $\Gr^{(1)}\DMF$ 
acts by the multiplication by $\alpha$ on $\SRM$. 
\end{ntt}

\begin{prop} \label{DR:GrDMF}
The graded module $\Gr^*\DMF$ over the graded ring $\GRMF \simeq \SRM$ is isomorphic to $\DSym$.
\end{prop}
\begin{proof}
It is a sub-module of the $R$-linear endomorphisms of $\Gr^*\RMF$ generated by the same elements as $\DSym$. 
\end{proof}

\begin{ntt}\label{DR:filtdef}
Let $A$ be a ring and let $\mI$ be an ideal of $A$, 
such that $A$ is complete and Hausdorff for the $\mI$-adic topology. 
Let $M$ be an $A$-module together with a $\mI$-filtration 
$(M^{(i)})_{i\in \ZZ}$, \ie $\mI M^{(i)} \subseteq M^{(i+1)}$ for any $i$. 
Consider the associated graded ring $\Gr^* A$ and the associated
graded $\Gr^* A$-module $\Gr^* M$.

The filtration is called {\em exhaustive} 
(see \cite[III, \S2, 1]{Bourbaki70})
if any $\lambda \in M$ belongs to some 
$M^{(i)}$, \ie $M=\cup_i M^{(i)}$.
For an exhaustive filtration
let $\nu(\lambda)$ denote the largest integer such that 
$\lambda \in M^{\nu(\lambda)}$ and let $\bar \lambda$ 
denote the class of $\lambda$ in 
$M^{(\nu(\lambda))}/M^{(\nu(\lambda)+1)}$ if $\nu(\lambda)$ exists and 
$0$ if $\nu(\lambda)$ does not exist. 
Note that
$\bar \lambda$ is nonzero if and only if $\lambda$ is not in $\cap_i M^{(i)}$.
\end{ntt}

\begin{lem} \label{DR:liftings}
Let $M$ be an $A$-module with an exhaustive $\mI$-filtration. Let $\lambda_1, \ldots , \lambda_n$ be elements of $M$. Then
\begin{enumerate}
\item \label{DR:generate} 
Assume $A$ is complete for the $\mI$-adic topology and 
$M$ is Hausdorff. 
If $\bar \lambda_1, \ldots, \bar \lambda_n$ generate $\Gr^* M$ as a 
$\Gr^* A$-module, then $\lambda_1, \ldots, \lambda_n$ 
generate $M$ as an $A$-module.
\item \label{DR:independent} 
Assume $A$ is Hausdorff. 
If $\bar \lambda_1, \ldots, \bar \lambda_n$ are independent elements in the $\Gr^* A$-module $\Gr^* M$, then $\lambda_1, \ldots, \lambda_n$ are independent in the $A$-module $M$.
\item \label{DR:free} 
Assume $A$ is Hausdorff complete and 
$M$ is Hausdorff.
If $\Gr^* M$ is a finitely generated free $\Gr^* A$-module, 
then $M$ is a finitely generated free $A$-module. 
\end{enumerate}
\end{lem}

\begin{proof}
Let $A[i]$ be $A$ itself considered as an 
$A$-module with the shifted 
filtration $\mI^{j+\nu(\lambda_i)}$
(by convention, $\mI^j=A$ for $j\leq 0$). 
The claims \eqref{DR:generate} and \eqref{DR:independent} then 
follow from \cite[III, \S2, 8, Cor.~1 and 2]{Bourbaki85} 
applied to $X=\bigoplus_i A[i]$, $Y=M$ and $\bigoplus_i A[i] \to M$ 
the map sending $(a_i)_i$ to $\sum_i a_i \lambda_i$. 
The last claim \eqref{DR:free} is an immediate consequence
of \eqref{DR:generate} and \eqref{DR:independent}.
\end{proof}

We now come to the main result of this section:

\begin{thm} \label{DR:basisDMF}
Let $\DMF$ be the algebra of operators defined above.
For each element $w \in W$, choose a reduced decomposition 
$I_w$ in simple reflections.
Then the operators $\Del_{I_w}^F$ (resp. $\CC_{I_w}^F$) 
form a basis of $\DMF$ as an $\RMF$-module.
\end{thm}

\begin{proof}
When the formal group law is additive, this is proved in \cite[\S 4, Cor. 1]{Demazure73} for the $\Del^{F_a}_w$ (and $\CC^{F_a}_w=(-1)^{l(w)}\Del^{F_a}_w$).
The filtered module $\DMF$ is Hausdorff by point \eqref{DR:DMFHausdorff_item} in Proposition \ref{DR:DMFFil}, and $\RMF$ is complete. We can therefore apply point \eqref{DR:free} of Lemma~\ref{DR:liftings} to deduce the general case from the corresponding fact on the associated graded objects, which is the additive case by Proposition \ref{DR:GrDMF}.
\end{proof}


\section{Torsion indices and augmented operators} \label{IT}

In the present section we study the augmentation $\eDMF$
of the algebra of operators $\DMF$. 
The main result is Proposition~\ref{DR:basiseDel}
which is what becomes of Theorem~\ref{DR:basisDMF} when the augmentation $\aug$ is applied.
Similarly, we then examine the filtration on $\eDMF$ obtained by applying $\aug$ to the one on $\DMF$ introduced in \ref{DR:DMFfil}.

\begin{ntt}
Consider the {\em torsion index} $\tor$ of the root system 
as defined in \cite[\S 5]{Demazure73} 
and its prime divisors, called {\em torsion primes}. 
It has been computed for simply connected root systems of all types
(see \cite[\S 7, Prop. 8]{Demazure73} and \cite{Totaro05b,Totaro05}).
The results of these computations are summarized in the following table:
\medskip

{\small
\begin{center}
\begin{tabular}{l||c|c|c|c|c|c|c|c|c}
Type & $A_l$ & $B_l$, $l\geq 3$ & $C_l$ & $D_l$, $l\geq 4$ & $G_2$ & $F_4$ & $E_6$ & $E_7$ & $E_8$ \\
Torsion primes & $\emptyset$ & $2$ & $\emptyset$ & $2$ & $2$ & $2,3$ & $2,3$ & $2,3$ & $2,3,5$ \\
Torsion index & $1$ & $2^{e(l)}$ & $1$ & $2^{e(l-1)}$ & $2$ & $2\cdot 3$ & $2\cdot 3$ & $2^2\cdot 3$ 
& $2^6\cdot 3^2\cdot 5$
\end{tabular}
\end{center}
}
\medskip

\noindent The exponent $e(l)$ is equal to 
$l-\lfloor\log_2(\binom{l+1}{2}+1)\rfloor$ 
except for certain values of $l$ 
equal or slightly larger than a power of $2$ which are 
explicitly given in \cite[Thm. 0.1]{Totaro05}. 
Since $\tor(\mathcal R_1\times \mathcal R_2)=\tor(\mathcal R_1)\cdot \tor(\mathcal R_2)$, we may compute
torsion indices of semi-simple root systems by reducing to irreducible ones.
\end{ntt}

\begin{ntt}
Let $M$ be the character group of a root system of $G$ and let
$F$ be a formal group law over a ring $R$.
Let $\aug\colon \RMF\to R$ be the augmentation map with the kernel $\mI_F$.
Let $N$ be the length of the longest element $w_0$ 
of the Weyl group $W$ of $G$. 
Let $F_a$ be an additive formal group law.

By definition of the torsion index (see \cite[\S 5]{Demazure73})
there exists an element $a$ of the symmetric algebra 
$\Sym{N}{R}{M}$ such that 
$$
\Del^{F_a}_{w_0}(a)=\tor.
$$
Let $\phi\colon S^*_R(M)\stackrel{\simeq}\to \GRMF$ be an isomorphism
from Lemma~\ref{DR:SRM=GRMF}.
Let $u_0\in \mI^N_F$ be an element such that $\bar u_0=\phi(a)$. 
Therefore, 
for any reduced decomposition $I_0$ of $w_0$ 
we have $\aug \Del_{I_0}^F(u_0)=\tor$ or, in other words, by Prop.~\ref{DR:GrDel}
$$
\Del_{I_0}^F(u_0)=\tor+\mI_F \quad \text{and} 
\quad \CC_{I_0}^F(u_0)=(-1)^N \tor+\mI_F \quad \text{in}\ \RMF. 
$$
\end{ntt}

\begin{lem} \label{DR:reducedI0}
Let $I$ be a sequence of simple reflections and 
let $l(I)$ be its length. Then 
\begin{enumerate}
\item \label{DR:Delpull_item} 
for any $i\in \ZZ$, $\Del_I^F(\mI^i_F)\subseteq \mI^{i-l(I)}_F$ and $\CC_I^F(\mI^i_F)\subseteq \mI^{i-l(I)}_F$;
\item \label{DR:notred_item} 
if $l(I)\leq i$ and $I$ is not reduced, 
then 
$$\Del_I^F(\mI^i_F) \subseteq \mI^{i-l(I)+1}_F \qquad\text{and}\qquad \CC_I^F(\mI^i_F) \subseteq \mI^{i-l(I)+1}_F;$$
\item \label{DR:reducedI0_item} 
let $u_0\in\mI_F^N$ be the element chosen above.
If $l(I)\leq N$ then 
$$
\aug\Del_{I}^F(u_0)=(-1)^N \aug\CC_{I}^F(u_0)=
\begin{cases}
\tor & \text{if $I$ is reduced and $l(I)=N$} \\
0 & \text{otherwise.}
\end{cases}
$$
\end{enumerate}
\end{lem}

\begin{proof}
Claim \eqref{DR:Delpull_item} follows from 
Prop.~\ref{DR:DMFFil}, \eqref{DR:Delup_item}. 
To check \eqref{DR:notred_item}, it therefore suffices 
to check that $\gDel_I^F$ is zero. 
It follows from the corresponding fact in the additive case
\cite[\S 4, Prop. 3, (a)]{Demazure73} 
after applying Prop.~\ref{DR:GrDel}. 
The ``otherwise'' case of \eqref{DR:reducedI0_item} 
follows from \eqref{DR:notred_item}. 
The case where $I$ is reduced and has length $N$ 
follows from the definition of $u_0$. 
\end{proof}

\begin{prop} \label{DR:basiseDel}
Consider the $R$-module $\eDMF$ of all $R$-linear forms $\aug D$, where $D\in\DMF$.
Assume the torsion index $\tor$ is regular in $R$. 
Then for any choice of a collection of reduced decomposition $I_w$ 
for each $w \in W$, the $(\aug \Del_{I_w})_{w \in W}$ 
(resp.~the $(\aug \CC_{I_w})_{w \in W}$) form an $R$-basis of $\eDMF$.  
\end{prop}

\begin{proof}
By Theorem \ref{DR:basisDMF}, the $\aug\Del_{I_w}$ generate $\eDMF$. We prove that the $\aug\Del_{I_w}$ such that $l(I_w)\leq i$ are independent by induction on $i$. It is obviously true when $i<0$.
Let $\sum_{l(w)\leq i} r_w \aug \Del_{I_w}=0$. 
By induction, it suffices to show that the coefficients $r_w$ with $l(w)=i$ are zero. 
For any $v$ and $w$ of length $i$, the concatenation $I_w+I_{v^{-1}w_0}$ is a reduced decomposition of $w_0$ if and only if $v=w$. 
Thus, by Lemma \ref{DR:reducedI0}.\eqref{DR:Delpull_item} and
\eqref{DR:reducedI0_item} we have $\sum_{l(w)\leq i} r_w \aug \Del_{I_w}(\Del_{I_{v^{-1}w_0}}(u_0))= r_v \tor$ which implies that $r_v=0$ by regularity of $\tor$.
The same proof works for the $\CC_{I_w}$.
\end{proof}

\begin{lem} \label{DR:DMFFilrevis_lem}
Assume $\tor$ is regular in $R$. 
\begin{enumerate}
\item \label{DR:DelIwinlw_item}  
For any reduced decomposition $I_w$ of $w$ and any element 
$u \in \mI^i_F\smallsetminus \mI^{i+1}_F$, the element $u \Del_{I_w}^F$ is in 
$\DMF^{(i-l(w))}\smallsetminus \DMF^{(i-l(w)+1)}$. 
\item \label{DR:DMFiis_item} 
The group $\DMF^{(i)}$ is the set of $D \in \DMF$ such that 
$D(\mI^j_F)\subseteq \mI^{i+j}_F$.
\end{enumerate}
\end{lem}

\begin{proof}
It suffices to prove \eqref{DR:DelIwinlw_item} 
after inverting the torsion index. 
By Lemma~\ref{DR:reducedI0}.\eqref{DR:Delpull_item}, 
the operator $u\Del_{I_w}^F$ is in $\DMF^{(i-l(w))}$.
Consider the element 
$y=\Del_{I_{w^{-1}w_0}}^F(u_0)$. It is in $\mI^{l(w)}_F$.
By Lemma~\ref{DR:reducedI0}.\eqref{DR:reducedI0_item} the
element $s=\Del_{I_{w}}^F(y)$ is invertible.
Hence, $u\Del_{I_w}^F(y)=us$ cannot be in $\mI^{i+1}_F$, 
and therefore $u \Del_{I_w}^F$ is not in $\DMF^{(i-l(w)+1)}$.

To prove \eqref{DR:DMFiis_item}, 
first note that 
Prop.~\ref{DR:DMFFil}.\eqref{DR:Dfiltrpres_item} immediately
gives one inclusion. 
Suppose $D$ is such that $D(\mI^j_F)\subseteq \mI^{i+j}_F$.
Using Theorem~\ref{DR:basisDMF} we may 
write it as 
$$
D=\sum_w u_w \Del_{I_w}^F.
$$
Recall that $\nu(u_w)$ denotes the largest integer such that 
$u_w \in \mI^v_F$ (see \ref{DR:filtdef}). 
Consider the set $S$ of $w\in W$ such that 
$\nu(u_w)-l(w)$ is minimal and take an element 
$w_1 \in S$ such that $l(w_1)$ is maximal. 
To show that 
$D \in \DMF^{(i)}$, it suffices to show that $\nu(u_{w_1})-l(w_1) \geq i$ by minimality of elements of $S$.

Let $y_1=\Del_{I_{w_1^{-1}w_0}}^F(u_0)$ be the corresponding element in $\mI^{l(w_1)}$.
We have
$$
D(y_1)= u_{w_1} \Del_{I_{w_1}}^F(y_1) + 
\sum_{w \in S\smallsetminus\{w_1\}} u_w \Del_{I_w}(y_1)
+\sum_{w \notin S} u_w \Del_{I_w}(y_1).
$$
By Lemma~\ref{DR:reducedI0}.\eqref{DR:Delpull_item} 
and \eqref{DR:notred_item}, 
two sums on the right are in $\mI^{\nu(u_{w_1})+1}_F$. 
As in the proof of \eqref{DR:DelIwinlw_item}, $u_{w_1}\Del_{I_{w_1}}^F(y_1)$ is not in $\mI^{\nu(u_{w_1})+1}$ so
$D(y_1)$ is in $\mI^{\nu(u_{w_1})}_F\smallsetminus\mI^{\nu(u_{w_1})+1}_F$. 
Since $D(y_1)$ is also in $\mI^{l(w_1)+i}_F$, we obtain that 
$\nu(u_{w_1})-l(w_1)\geq i$.
\end{proof}

We now consider the filtration on $\eDMF$ 
that is the image by $\aug$ 
of the filtration on $\DMF$ defined in \ref{DR:DMFfil}, \ie
$$
\eDMF \supseteq \cdots \supseteq \eDMF^{(i)} 
\supseteq \eDMF^{(i+1)} \supseteq \cdots \supseteq 0.
$$

\begin{prop} \label{PS:filtrationeDMF_prop}
Assume $\tor$ is regular in $R$. 
This filtration satisfies the following.
\begin{enumerate}
\item \label{PS:eDizeroIi1_item} 
For any $i \in \ZZ$ and $\aug D \in \eDMF^{(i)}$, 
we have $\aug D(\mI^{-i+1}_F)=0$.
\item \label{PS:eDMFquotientbasis_item} 
Let $E$ be a subset of $W$ of elements of length $l$, 
and let the $I_w$ be reduced decompositions of each of the $w\in E$. 
Then any nonzero $R$-linear combination 
$$
f=\sum_{w \in S} r_w\cdot\aug\Del_{I_w}^F \quad\text{(resp. $\aug C_{I_w}^F)$}
$$ 
is in $\eDMF^{-l}\smallsetminus \eDMF^{-l+1}$.
In particular, for any $w$ 
$$
\aug \Del_{I_w}^F 
\in\eDMF^{-l(w)}\smallsetminus\eDMF^{-l(w)+1}
\quad\text{(resp. with $C_{I_w}^F$).}
$$
\item \label{PS:eDMFibasis_item} 
For any choice of reduced decompositions $I_w$ 
for every element $w \in W$, 
the $R$-module $\eDMF^{(-i)}$ is 
a free $R$-module with basis the $\aug \Del_{I_w}$ (resp. $\aug C_{I_w}$)  
with $l(w) \leq i$. 
\item \label{PS:eDMFibounded_item} 
For any $i\leq -N$, we have $\eDMF^{(i)}=\eDMF$ and 
for any $i>0$, we have $\eDMF^{(i)}=0$, 
\ie the filtration is of the form
$$\eDMF = \eDMF^{(-N)} \supseteq \eDMF^{(-N+1)} 
\supseteq \cdots \supseteq \eDMF^{(0)} \supseteq 0.$$ 
\end{enumerate}
\end{prop}

\begin{proof}
Claim \eqref{PS:eDizeroIi1_item} follows from Lemma~\ref{DR:DMFFilrevis_lem}.\eqref{DR:DMFiis_item}.
To prove \eqref{PS:eDMFquotientbasis_item},
we first apply $f$ to $\Delta_{I_{w^{-1}w_0}}^F(u_0)$. 
By Lemma \ref{DR:reducedI0}.\eqref{DR:reducedI0_item}, 
we obtain $\tor r_w$, 
which has to be zero by \eqref{PS:eDizeroIi1_item} 
for the element $f$ to be in $\mI^{l+1}_F$. 
Therefore, each $r_w$ is zero.  
The last two claims follow.
\end{proof}


\section{Invariants and the characteristic map}\label{IC}

We now come to the definition of an algebraic replacement $\HMF$ 
for the oriented cohomology $\HH^*(G/B)$ 
when $G$ is a split semi-simple simply connected algebraic group 
corresponding to the root system and 
$B$ is a Borel of $G$. 
The identification of $\HMF$ with $\HH^*(G/B)$ is the subject of section \ref{AG}. 

\begin{dfn} \label{DR:charmap_def}
Let $\HMF$ be the $R$-dual of $\eDMF$ and let
$$\cc^F:\RMF \to \HMF$$
denote the natural map obtained by duality, 
\ie sending $u$ to the evaluation at $u$. When the formal group law is clear from the context, we write $\cc$ for $\cc^F$.

Again, when $F=F_a$ is additive, 
this $\HMF$ corresponds to the one 
defined by Demazure in \cite[\S 3]{Demazure73}.
\end{dfn}

\begin{rem} \label{DR:ccepsW}
Observe that when $R$ has no $2$-torsion, since the operators $\Del_\alpha^F$ are $\RMF$-linear by Proposition \ref{DO:DelWlin}, the characteristic map satisfies $\cc(fu)=\aug(f)\cc(u)$ when $f$ is in $\RMF^W$. 
\end{rem}

\begin{thm} \label{DR:charMapBasis}
Let $(I_w)_{w \in W}$ be a choice of reduced decompositions.
When the torsion index $\tor$ is regular in $R$, there is a unique $R$-basis $\zz^{\Del}_{I_w}$ of $\HMF$ such that the characteristic map is given by
$$
\cc(u)=\sum_w \aug\Del_{I_w}^F(u)\zz^{\Del}_{I_w}.
$$
Similarly, there is a unique $R$-basis 
$\zz^{\CC}_{I_w}$ of $\HMF$ such that the characteristic map is given by
$$\cc(u)=\sum_w \aug\CC_{I_w}^F(u)\zz^{\CC}_{I_w}.$$
\end{thm}
\begin{proof}
We take the bases that are dual to the ones of Proposition \ref{DR:basiseDel}. 
\end{proof}

\begin{thm} \label{DR:surjchar}
Assume that the torsion index $\tor$ is invertible in $R$. Then for any choice of a collection of reduced decompositions $I_w$ for every $w \in W$, the $\cc\big(\Del_{I_w}(u_0)\big)$ (resp.~the $\cc\big(\CC_{I_w}(u_0)\big)$) form an $R$-basis of the image of $\cc$. In particular, the characteristic map $\cc$ is surjective.
\end{thm}
\begin{proof}
We have $\cc(u_0)=\tor . \zz^{\Del}_{I_0}$ and 
$$\cc\big(\Del_{I_{w^{-1}w_0}}(u_0)\big)=\tor . \zz^{\Del}_w + \sum_{l(v)>l(w)} \aug \Del_{I_v}\Del_{I_{w^{-1}w_0}}(u_0)\zz^{\Del}_{I_v}$$
by point (3) of Lemma \ref{DR:reducedI0}. All the $\zz^\Del_{I_w}$ are thus in the image of the characteristic map by decreasing induction on the length of $w$. The same proof works when replacing $\Del$ by $\CC$. 
\end{proof}

Still assuming that $\tor$ is regular in $R$, let $f:R \to R'$ be a ring morphism sending a formal group law $F$ over $R$ to a formal group law $F'$ over $R'$. By restriction (through $f$), $\HFG{R'}{M}{F'}$ is a $R$-module. Let $f_\mH: \HMF \to \HFG{R'}{M}{F'}$ be the $R$-linear map sending the element $\zz_{I_w}^\Del$ in $\HMF$ to the corresponding one in $\HFG{R'}{M}{F'}$.
\begin{prop} \label{fHfirst_prop}
We have the following:
\begin{enumerate}
\item \label{fHfunct_item} The assignment $f \to f_\mH$ defines a functor from the category of formal group laws $(R,F)$ ($\tor$ regular in $R$) to the category of pairs $(R,M)$ where $M$ is a module over $R$, and the morphisms are the obvious ones. 
\item \label{fHc=cfH_item} We have $f_\mH \cc^F = \cc^{F'} f_*$.
\item \label{fHindepIw_item} The map $f_\mH$ defined above is independent of the choice of the reduced decompositions $I_w$. 
\end{enumerate}
\end{prop}
\begin{proof}
Point \eqref{fHfunct_item} holds by construction. 
Point \eqref{fHc=cfH_item} follows from the formula for 
the characteristic map given in Theorem \ref{DR:charMapBasis}. 
Since $f_*$ is defined independently of the $I_w$, 
the last item \eqref{fHindepIw_item} holds if the characteristic map is surjective. 
The general case follows by embedding 
$\HFG{R}{M}{F}$ into $\HFG{R\left[1/\tor\right]}{M}{F}$ 
where the characteristic map is surjective by Theorem \ref{DR:surjchar}.
\end{proof}

\begin{prop} \label{DR:system} Assume that $\tor$ is invertible in $R$.
Fix a reduced decomposition $I_w$ for each $w\in W$. Then
\begin{enumerate}
\item \label{DR:system_item} For any $x \in \RMF$, the system of linear equations in $\RMF$ 
$$\Del_{I_v}(x) = \sum_{w \in W} r_w \Del_{I_v} \Del_{I_w}(u_0)$$
for all $v$ in $W$, has a unique solution $(r_w)_{w \in W}$.
\item \label{DR:solutionbasis_item} If $(r_w)_{w \in W}$ is the solution of \eqref{DR:system_item}, then for any $D \in \DMF$, we have the equality
$$D(x) = \sum_{w \in W} r_w D \Del_{I_w}(u_0).$$
\item \label{DR:coeffsfixedW_item} If $(r_w)_{w\in W}$ is the solution of \eqref{DR:system_item} and if $R$ has no $2$-torsion, then all the $r_w$ are in fact in $\RMF^W$.
\end{enumerate}
The same is true when replacing $\Del$ by $\CC$ everywhere in points \eqref{DR:system_item}, \eqref{DR:solutionbasis_item} and \eqref{DR:coeffsfixedW_item}.
\end{prop}
\begin{proof}
Let $A$ be a ring and $\mI$ be an ideal of $A$ 
contained in its Jacobson radical. 
A matrix with coefficients in $A$ is invertible 
if and only if 
the corresponding matrix with coefficients in $A/\mI$ is invertible. 
Since $\mI_F$ is contained in the radical of $\RMF$, 
it suffices to show that the matrix of the system in $\RMF/\mI_F$ is invertible. 
If we order the $v$'s by increasing length and 
the $w$'s by decreasing length, 
Lemma \ref{DR:reducedI0} shows that 
the matrix is lower triangular with $\tor$ on the diagonal, 
and it is therefore invertible. 
Point \eqref{DR:solutionbasis_item} follows from point \eqref{DR:system_item} since the $\Del_{I_w}$ (resp. the $\CC_{I_w}$) form a basis of $\DMF$ as an $\RMF$-module by Theorem \ref{DR:basisDMF}. Let us prove point \eqref{DR:coeffsfixedW_item}. For any simple root $\alpha$ and for any $v \in W$, we have
\begin{equation*}
\begin{split}
\Del_\alpha \Del_{I_v}(x) & = \sum_{w \in W} \Del_\alpha \big(r_w \Del_{I_v} \Del_{I_w}(u_0)\big) \\
& = \sum_{w \in W} r_w \Del_\alpha \Del_{I_v} \Del_{I_w}(u_0)  + \sum_{w \in W} \Del_\alpha(r_w) s_\alpha \Del_{I_v} \Del_{I_w}(u_0) 
\end{split}
\end{equation*}
using point \eqref{DO:DelDeriv_item} of Proposition \ref{DO:eqDel}.
But we also have
$$\Del_\alpha \Del_{I_v}(x) =\sum_{w \in W} r_w \Del_\alpha \Del_{I_v} \Del_{I_w}(u_0)$$ 
by point \eqref{DR:solutionbasis_item} with $D=\Del_{\alpha}\Del_{I_v}$.
So 
$$\sum_{w \in W} \Del_{\alpha}(r_w) s_{\alpha}\Del_{I_v} \Del_{I_w}(u_0) =0$$
for any $v \in W$. 
Applying the ring automorphism $s_\alpha$, we obtain that the $s_\alpha\Del_\alpha(r_w)$ are solution of the system \eqref{DR:system_item} with $x=0$ and are therefore $0$ by uniqueness. 
Thus $\Del_\alpha(r_w)=0$ and $r_w$ is fixed by $s_\alpha$ for any simple root $\alpha$ by Lemma \ref{DO:Delalpha0_lem} hence by the whole Weyl group $W$, since it is generated by the simple reflections.  

Point \eqref{DR:coeffsfixedW_item} for $\CC$ is proved exactly in the same way (still using $\Del_\alpha$, not $\CC_\alpha$).
\end{proof}

\begin{thm}
Assume that the torsion index $\tor$ is invertible in $R$ and that $R$ has no $2$-torsion. Choose a reduced decomposition $I_w$ for every $w \in W$. The elements $\Del_{I_w}(u_0)$ (resp. $\CC_{I_w}(u_0)$) form a basis of $\RMF$ as an $\RMF^W$-module.  
\end{thm}
\begin{proof}
We need to show that any $x$ can be decomposed in a unique way as $x=\sum_{w \in W} r_w \Del_{I_w} (u_0)$. 
Note that this is the row $v=1$ of the system \eqref{DR:system_item} of Proposition \ref{DR:system}. 
By $\RMF^W$-linearity of any $\Del_{I_v}$ (see Prop.~\ref{DO:DelWlin}), if that row is satisfied with coefficients $r_w$ in $\RMF^W$, the rest of the system is satisfied, so this proves uniqueness. Existence of the decomposition then follows from points \eqref{DR:system_item} and \eqref{DR:coeffsfixedW_item} of the proposition.
The same proof goes through with the $\CC_{I_w}(u_0)$.
\end{proof}

\begin{rem}
The previous theorem for the (uncompleted) symmetric algebra, \ie the additive case, is \cite[\S 6, Th\'eor\`eme 2, (c)]{Demazure73}, but the proof given there is incorrect: in the notation of \loccit, the ideal $I$ is only known a priori to be of the right form to apply the graded Nakayama lemma when it is tensored by $\mathbb{Q}$, a fact that cannot be assumed, this is the whole point of the theorem. 
When contacted by one of the authors, Demazure kindly and quickly supplied another proof which we adapted to our setting in Proposition \ref{DR:system}. 
There is a slight difference: in the symmetric algebra case, the matrix of the system is upper triangular with diagonal $\tor$ and is therefore invertible, whereas in our case, the lower triangular part does not vanish because our formal group law is not additive, but the strictly upper triangular part is in the radical because our ring is complete. 
\end{rem}

\begin{thm} \label{DR:kercharmap_thm}
Assume that the torsion index $\tor$ is invertible in $R$ and that $R$ has no $2$-torsion. 
Then the kernel of the characteristic map $\cc: \RMF \to \HMF$ is the ideal of $\RMF$ generated by elements in $\mI^W_F$.
\end{thm}
\begin{proof}
By Remark \ref{DR:ccepsW}, the ideal generated by $\mI^W_F$ 
is included in $\ker \cc$.
Conversely, let $x\in \ker \cc$, and decompose it as $\sum_{w \in W} r_w \Del_{I_w}(u_0)$ with the $r_w \in \RMF^W$ by the previous theorem. 
We then have $0=\cc(x)=\sum_{w \in W}\aug(r_w) \cc\big(\Del_{I_w}(u_0)\big)$. 
But the $\cc\big(\Del_{I_w}(u_0)\big)$ form an $R$-basis of the image of $\cc$ by Theorem \ref{DR:surjchar}, so $\aug(r_w)=0$ for all $w \in W$ and $x$ is the ideal generated by $\mI^W_F$.
\end{proof}


\section{The product structure of the cohomology} \label{PS}

In this section, we explain how to define a product on $\HMF$ in order that the characteristic map be a ring homomorphism. Then, we study the structure of $\HMF$.

\begin{lem} \label{PS:productDdecomp_lem}
For any $D \in \DMF$, 
\begin{enumerate}
\item \label{PS:productDdecomp_item} there is a finite family $(D_i,D'_i)$ of elements of $\DMF$ such that $D(uv) = \sum_i D_i(u) D'_i(v)$ 
for any elements $u,v \in \RMF$. 
\item \label{PS:productFiltr_item} 
Furthermore if $D$ is in $\DMF^{(j)}$, each $D_i$ and $D'_i$ 
can be chosen in $\DMF^{(m_i)}$ and $\DMF^{(m_i')}$ respectively, such that $m_i + m_i' \geq j$. 
\end{enumerate}
\end{lem}

\begin{proof}
To show \eqref{PS:productDdecomp_item}, 
the elements of $\DMF$ satisfying a statement 
form a sub-$\RMF$-algebra of $\DMF$, 
and it is clear for generators of $\DMF$ 
by equation \eqref{DO:DelDeriv_item} of Prop.~\ref{DO:eqDel}. 
This same equation also proves 
\eqref{PS:productFiltr_item} for generators by induction.
\end{proof}

Let $\RMFod$ be the continuous $R$-dual of $\RMF$, 
\ie the set of $R$-linear morphisms $f$ from $\RMF$ to $R$ 
such that $f(\mI^i_F)=0$ for some $i$ (depending on $f$).

\begin{lem} \label{PS:flatduals_lem}
The $R$-module $\RMFod$ is flat.
When $\tor$ is regular in $R$, the $R$-module $\eDMF$ is also flat. 
\end{lem}

\begin{proof}
For any $i$, the $R$-module $\mI^i_F/\mI^{i-1}_F$ 
is a free $R$-module since it is isomorphic to $\Sym{i}{R}{M}$. 
Thus, by induction, for any $i$, the quotient $\RMF/\mI^i_F$ 
is a finitely generated free $R$-module, thus flat. 
Now $\RMFod$ is the direct limit of the $(\RMF/\mI^i)^\vee$ 
(usual $R$-duals) and it is therefore flat.
When $\tor$ is regular, 
the module $\eDMF$ is a finitely generated free $R$-module 
by \ref{DR:basiseDel}, so it is flat. 
\end{proof}

\begin{thm} \label{PS:productH_thm}
Let $\tor$ be regular in $R$. 
There is a unique ring structure on $\HMF$ 
such that 
the characteristic map $\cc: \RMF \to \HMF$ is a ring morphism.
\end{thm}

\begin{proof}
Since $R \to R[1/\tor]$ is an injection and $\HMF$ is a free $R$-module, 
we can check uniqueness after inverting $\tor$, 
in which case it is obvious since the characteristic map is surjective. 
Let us now prove the existence of this product.
The $R$-module $\RMFod$ has a natural structure of coalgebra 
induced by the collection of maps 
$(\RMF/\mI^i)^\vee \to (\RMF/\mI^i)^\vee \otimes (\RMF/\mI^i)^\vee$ 
dual to the product, and taking direct limits. 
The inclusion $\eDMF \subseteq \RMFod$ 
induces the diagonal map 
$\eDMF \otimes_R \eDMF \to \RMFod \otimes_R \RMFod$ 
which is injective by Lemma~\ref{PS:flatduals_lem}. 
To show that $\eDMF$ is a subcoalgebra of $\RMFod$, 
it therefore suffices to show that the composition 
$\eDMF \to \RMFod \to \RMFod \otimes_R \RMFod$ 
factors through the image of $\eDMF \otimes_R \eDMF$ in $\RMFod \otimes_R \RMFod$. 
This is ensured by Lemma~\ref{PS:productDdecomp_lem}. 
Now $\HMF$ is the dual of a coalgebra, 
and is therefore an $R$-algebra. 
By construction, the characteristic map is a ring homomorphism.
\end{proof}

\begin{ntt}
We now filter the ring $\HMF$ by setting that 
$\HMF^{(i)}$ is the set of forms on $\eDMF$ 
that are zero on $\eDMF^{(-i+1)}$. 
By Proposition~\ref{PS:filtrationeDMF_prop}, 
we therefore have $\HMF^{(0)}=\HMF$ and $\HMF^{(N+1)}=0$. 
In other words, the filtration has the form
$$\HMF=\HMF^{(0)}\supseteq \HMF^{(1)} 
\supseteq \cdots \supseteq \HMF^{(N)} \supseteq 0.$$
We consider the augmentation map $\augH:\HMF \to R$ 
defined as the evaluation at $\aug \in \eDMF$.
\end{ntt}

\begin{prop} \label{PS:Filcharmap_prop}
Assume $\tor$ is regular in $R$. We have:
\begin{enumerate}
\item \label{PS:HMF1aug_item} 
$\HMF^{(1)}$ is the kernel of $\augH$.
\item \label{PS:FilHMFprod_item} 
The filtration is compatible with the product: 
$$\HMF^{(j)}\HMF^{(m)} \subseteq \HMF^{(j+m)}$$
\item \label{PS:charmapfil_item} 
The characteristic map $\cc$ is a morphism of filtered rings.
\end{enumerate}
\end{prop}

\begin{proof}
Point \eqref{PS:HMF1aug_item} is clear since $\aug$ generates $\DMF^{(0)}$.
Let $h_1 \in \HMF^{(j)}$ and $h_2 \in \HMF^{(m)}$.
The product $h=h_1.h_2$ is defined as 
$h(\aug D) = \sum_i h_1(\aug D_i) h_2(\aug D'_i)$ 
where the $(D_i,D'_i)$ are as in Lemma \ref{PS:productDdecomp_lem}. 
By the same lemma, if $D \in \DMF^{(j+m+1)}$, 
they can be chosen such that $D_i \in \DMF^{j_i}$ and 
$D'_i \in \DMF^{m_i}$ with $j_i +m_i \geq j+m+1$. 
Thus, for any $i$, either $j_i \geq j+1$ or $m_i \geq m+1$, 
so every term in the sum is zero and $h \in \HMF^{(i+j)}$. 
For point \eqref{PS:charmapfil_item}, take an element $u$ in $\mI^i_F$. 
By definition, $\cc(u)$ is the evaluation at $u$ in $\eDMF^\vee=\HMF$ 
so by Prop.~\ref{PS:filtrationeDMF_prop}.\eqref{PS:eDizeroIi1_item}, 
$\cc(u)$ is in $\HMF^{(i)}$.
\end{proof}

\begin{prop} \label{PS:basisHMFi_prop}
Let $I_w$ be choices of reduced decompositions for every $w \in W$. Assume the torsion index $\tor$ is regular in $R$. Then
\begin{enumerate}
\item \label{PS:basisHMFi_item} For any $i$, the elements $\zz^\Del_{I_w}$ (resp.~$\zz^\CC_{I_w}$)  with $l(w)\geq i$ generate $\HMF^{(i)}$. 
\item The elements $\zz^\Del_{I_{w_0}}$ and $\zz^{\CC}_{I_{w_0}}$ of Theorem \ref{DR:charMapBasis} do not depend on the choices of decompositions $I_w$ and are equal up to the sign $(-1)^N$. The element $\zz^{\CC}_{I_{w_0}}$ is denoted by $\zz_0$ and for any $u_0$ as in section \ref{IT}, we have $\cc(u_0)=(-1)^N \tor \zz_0$.
\end{enumerate} 
\end{prop}
\begin{proof}
Point \eqref{PS:basisHMFi_item} follows from point \eqref{PS:eDMFibasis_item} of Proposition \ref{PS:filtrationeDMF_prop} since the $\zz^{\Del}_{I_w}$ (resp.~ $\zz^{\CC}_{I_w}$) are the dual basis to the $\aug \Del_{I_w}$ (resp.~$\aug \CC_{I_w}$).
By Theorem \ref{DR:charMapBasis}, we have $\cc(u_0)=\tor \zz_{I_{w_0}}^{\Del}$. But the left hand side is independent of the choices of the $I_w$ and the right hand side is independent of the choice of $u_0$, which proves the claim. We also have $\cc(u_0)=(-1)^N \tor \zz_0$.
\end{proof}

\begin{ntt}
We now consider operators on $\HMF$. For any root $\alpha$, multiplication on the right by $\CC_{\alpha}$ (resp.~$\Del_{\alpha})$) defines an endomorphism on $\eDMF$, and therefore one on $\HMF$ by duality. This operator is denoted by $\Aa_\alpha$ (resp.~$\Bb_\alpha$). In other words, for any $h \in \HMF$ and $\aug D \in \eDMF$, we have
$$\Aa_\alpha(h)(\aug D) = h(\aug D \CC_\alpha) \qquad \text{and} \qquad \Bb_\alpha(h)(\aug D) = h(\aug D \Del_\alpha)$$ 
As for the operators $\Del$ and $\CC$, we use the notation $\Aa_I$ and $\Bb_I$ for a sequence $I$ of simple reflections. 
\end{ntt}

\begin{prop} \label{PS:AandB_prop}
When $\tor$ is regular in $R$, these operators satisfy:
\begin{enumerate}
\item \label{PS:AcBc_item} $\Aa_\alpha \circ \cc (u) = \cc \circ \CC_\alpha(u)$ and $\Bb_\alpha \circ \cc (u) = \cc \circ \Del_\alpha(u)$.
\item \label{PS:AandBup_item} $\Aa_\alpha \HMF^{(i)} \subseteq \HMF^{(i-1)}$ and $\Aa_\alpha \HMF^{(i)} \subseteq \HMF^{(i-1)}$
\item \label{PS:AI0z0_item} For any reduced decomposition $I_0$ of $w_0$, the element $\Aa_{I_0}(\zz_0)$ (resp.~$\Bb_{I_0}(\zz_0)$) is invertible in $\HMF$. 
\item \label{Anonreduced_item} If $I$ is nonreduced and of length $N$, then the element $\Aa_{I}(\zz_0)$ (resp.~$\Bb_{I}(\zz_0)$) is in $\HMF^{(1)}=\ker \augH$. 
\end{enumerate}
\end{prop}
\begin{proof}
Point \eqref{PS:AcBc_item} follows directly from the definitions of the characteristic map and of the operators. Point \eqref{PS:AandBup_item} follows from point \eqref{DR:Delup_item} of Proposition \ref{DR:DMFFil}. For point \eqref{PS:AI0z0_item}, first note that by Proposition \ref{PS:Filcharmap_prop}, an element $h \in \HMF$ is invertible if and only if $\augH (h)$ is invertible in $R$. Now
$$\tor\;\augH\Aa_{I_0}(\zz_0)=\augH\Aa_{I_0}(\tor \zz_0) = \augH\Aa_{I_0}\cc(u_0)= \augH \cc\big( \CC_{I_0}(u_0)\big) = \aug \CC_{I_0}(u_0) = \tor.$$
Therefore, $\augH \Aa_{I_0}(\zz_0)=1$. The same proof works for $\augH \Bb_{I_0}(\zz_0)$.
For point \eqref{Anonreduced_item}, the same series of equalities is used, except that the last term is zero by Prop.~\ref{DR:reducedI0}. 
\end{proof}

\begin{prop} \label{PS:AIbasisHMF_prop}
Let $I_w$ be a choice of reduced decompositions for all $w \in W$. Assume $\tor$ is regular in $R$. Then the elements $\Aa_{I_w}(\zz_0)$ (resp.~$\Bb_{I_w}(\zz_0)$) with $l(w)\leq N-i$ form an $R$-basis of $\HMF^{(i)}$. 
\end{prop}
\begin{proof}
We prove it for the $\Aa_{I_w}(\zz_0)$, the proof for the $\Bb_{I_w}(\zz_0)$ is similar. First note that by Proposition \ref{PS:basisHMFi_prop}, $\zz_0$ is the unique element of $\HMF$ such that $\zz_0(\aug \CC_{I_w}) = \delta_{w,w_0}$. 
Thus, if $l(w)+l(v) \leq N$, we have $\Aa_{I_w}(\zz_0)\big(\aug \CC_{I_v}\big) = \delta_{w,v^{-1}w_0}$. 
In other words, if we decompose $A_{I_w}$ on the basis of the $\zz_{I_v}$, the coordinate on $\zz_{I_v}$ is $\delta_{w,v^{-1}w_0}$ if $l(v)+l(w) \leq N$. 
Thus, the $A_{I_w}(\zz_0)$ with $l(w) \leq N-i$ can be expressed as linear combinations of the $\zz_{I_v}$ with $l(v)\geq i$ and the matrix of their expressions is unitriangular up to a correct ordering of rows and columns and is therefore invertible. 
\end{proof}

\begin{prop} \label{fHringmor_prop}
Assume $\tor$ is regular in $R$ and $R'$, respectively endowed with formal group laws $F$ and $F'$. Let $f:R \to R'$ be a ring morphism sending $F$ to $F'$. Then the morphism $f_\mH: \HFG{R}{M}{F} \to \HFG{R'}{M}{F'}$ introduced before Proposition \ref{fHfirst_prop} is actually a morphism of filtered rings. It satisfies
$$\Aa_I^{F'} f_\mH = f_\mH \Aa_I^F\qquad \text{and}\qquad\Bb_I^{F'} f_\mH = f_\mH \Bb_I^F$$
for any sequence $I$.
\end{prop}
\begin{proof}
Both facts are clear by Point \ref{fHc=cfH_item} of \ref{fHfirst_prop} when the characteristic map is surjective, and therefore when $\tor$ is regular by extending scalars to $R\left[1/t\right]$. 
\end{proof}


\part{Bott-Samelson resolutions for oriented theories.}

\section{Algebraic cobordism and oriented cohomology theories}
\label{AC}

In the present section we recall 
the notion of {\em oriented} cohomology theory and 
the notion of algebraic cobordism $\Omega$
following the book of Levine and Morel \cite{Levine07}.
As main examples of oriented theories 
we consider the Chow ring $\CH$, Grothendieck's 
$\KT^0$ and connective $K$-theory $\KK$. 

\begin{ntt} 
According to \cite[\S1.1]{Levine07}
an {\em oriented cohomology theory} $\HH$ is 
a contravariant functor
from the category of smooth varieties over a field $k$ of an arbitrary characteristic
to the category of graded rings 
satisfying the standard cohomological axioms: 
localization, homotopy invariance and push-forwards
for projective morphisms (see \cite[Def.1.1.2]{Levine07}).

A universal theory $\Omega$ satisfying these properties 
was constructed in \cite[\S2]{Levine07} assuming the base field $k$
has characteristic $0$. 
It is called {\em algebraic cobordism}.
An element of codimension $i$ in $\Omega(X)$, 
\ie in $\Omega^i(X)$, 
is additively generated by classes $[Y\to X]$
of projective maps of codimension $i$ from smooth schemes $Y$. 

For any map $f\colon X_1\to X_2$ the induced functorial map
$f^*\colon \Omega^i(X_2)\to \Omega^i(X_1)$ 
is called the {\em pull-back}.
For a projective map $f\colon X_1\to X_2$, there is a corresponding {\em push-forward} map $f_*\colon \Omega(X_1)\to \Omega(X_2)$ shifting the cohomological degree by $\dim X_2-\dim X_1$ when $X_1$ and $X_2$ are equidimensional and given by
$[Y_1\to X_1]\mapsto [Y_1\to X_1\xra{f} X_2]$.
\end{ntt}

\begin{ntt} 
An oriented cohomology theory $\HH$ comes with a {\em formal group law} $F$ over the coefficient ring $\HH(\pt)$ such that 
$$
F(c_1^{\HH}(\Lb_1),c_2^{\HH}(\Lb_2))=c_1^{\HH}(\Lb_1\otimes \Lb_2),
$$
where $\Lb_1$ and $\Lb_2$ are lines bundles on $X$ and 
$c_1^{\HH}$ denotes the first Chern class 
in the cohomology theory $\HH$ (see \cite[Cor.4.1.8]{Levine07}).
For $\Omega$, the associated formal group law~$U$
$$
U(x,y)=x+y+\sum_{i,j\ge 1} a_{ij}x^iy^j,\quad\text{ where } 
a_{ij}\in\Omega(\pt).
$$ 
turns out to be universal. 
The ring of coefficients $\Omega(\pt)$ is generated 
by the coefficients $a_{ij}$ and 
coincides with the classical {\em Lazard ring} $\LL$.

There is a canonical map $pr_{\HH}\colon \Omega \to \HH$ sending the coefficients of $U$ to the corresponding coefficients of $F$, and hence inducing
a morphism of formal group laws.
\end{ntt}

\begin{ex}\label{AC:chow} 
Consider the Chow ring $\CH(X)$ of algebraic cycles on $X$ 
modulo rational equivalence.
According to \cite[Thm.~4.5.1]{Levine07} 
the canonical map $pr_{\CH}\colon \Omega\to\CH$ 
induces an isomorphism
$\Omega\otimes_\LL \ZZ\stackrel{\simeq}\to\CH$ 
of oriented cohomology theories.
In particular, $pr_{\CH}\colon\Omega\to \CH$ is surjective 
and its kernel is generated by $\LL_{>0}$, the subgroup of elements of positive dimension in the Lazard ring.
Observe that $pr_{\CH}$ restricted to the coefficient rings
$\LL\to \ZZ=\CH(\pt)$ coincides with the augmentation map.
The associated formal group law, denoted by $F_0$, is called 
the {\em additive formal group law} and 
is given by $F_0(x,y)=x+y$. 
\end{ex}

\begin{ex}\label{AC:K0} 
Consider the oriented cohomology theory 
$\KT(X)=\KT^0(X)[\beta,\beta^{-1}]$,
where $\KT^0(X)$ is the Grothendieck $\KT^0$ of $X$.
Observe that $\KT(\pt)=\ZZ[\beta,\beta^{-1}]$.
According to \cite[Cor.4.2.12]{Levine07} 
the canonical map $pr_{\KT}\colon \Omega\to \KT$ 
sending $[\PP^1]$ to $\beta$
induces an isomorphism
$\Omega\otimes_\LL \ZZ[\beta,\beta^{-1}]\to \KT$ 
of oriented cohomology theories.
The associated formal group law, denoted by $F_\beta$,
is called a {\em multiplicative periodic formal group law} and is given by $F_\beta(x,y)=x+y-\beta xy$. 
\end{ex}

\begin{ex}\label{AC:connK} 
Consider a ring homomorphism 
$\LL\to \ZZ[v]$ given by $[\PP^1]=-a_{11}\mapsto v$ and
$a_{ij}\mapsto 0$ for $(i,j)\neq (1,1)$.
We define a new cohomology theory, called a {\em connective $K$-theory}, 
by
$\KK=\Omega\otimes_\LL \ZZ[v]$ (see \cite[\S 4.3.3]{Levine07}). 
Its formal group law is denoted by $F_v$ and is given by
$F_v(x,y)=x+y-vxy$. Observe that contrary to Example~\ref{AC:K0}
the element $v$ is non-invertible in the coefficient ring $\KK(\pt)=\ZZ[v]$.
\end{ex}

\begin{rem} Avoiding the localization axiom one obtains
a more general notion of oriented cohomology
theory studied by Merkurjev in \cite{Merkurjev02}.
In particular, he showed that taking the so-called `tilde operation'
of $\CH$ one obtains a theory $\tilde\CH$, defined over a field of any characteristic, which 
serves as a good replacement for algebraic cobordism 
in all questions related with cohomological operations. Note that the Lazard ring $\LL$ injects into the coefficient ring $\tilde\CH(\pt)$.
\end{rem}

We come now to the following important property of 
an algebraic cohomology theory:

\begin{dfn} 
We say that a cohomology theory is {\em weakly birationally invariant}
if for
any proper birational morphism $f\colon Y\to X$
the push-forward of the fundamental class $f_*(1_Y)$ is invertible.
\end{dfn}

\begin{ex}
Any {\em birationally invariant} theory, \ie  such that $f_*(1_Y)=1_X$ 
for any proper birational $f\colon Y\to X$, is
weakly birationally invariant.
The Chow ring $\CH$ considered over an arbitrary field
and the $K$-theory $\KT$ considered over a field
of characteristic 0 
provide examples of birationally invariant theories.
According to \cite[Thm.4.3.9]{Levine07}
the connective $K$-theory $\KK$ defined over a field of characteristic 0 
is universal 
among all {\em birationally invariant} 
theories. 
In particular, the kernel of the canonical map $\LL\to \KK(\pt)$
is generated by differences of classes $[X]-[X']$, where
$X$ and $X'$ are birationally equivalent.
\end{ex}

\begin{lem}\label{AC:kernilp}
For any smooth variety $X$ over a field of characteristic 0, the kernel of the natural map $\Omega^n(X)\to\CH^n(X)$ consists of nilpotent elements if $n\ge 0$ and is trivial if $n=\dim X$.
\end{lem}

\begin{proof}
According to \cite[Rem.~4.5.6]{Levine07} 
the kernel is additively generated by products of the form $ab$, $a\in\LL^{-i}$, $b\in\Omega^{n+i}(X)$, $i>0$. 
Now the claim follows from the fact that $\Omega^j(X)=0$ when $j>\dim X$.
\end{proof}

\begin{cor}\label{AC:pushinv}
Over a field of characteristic $0$, any oriented cohomology theory in the sense of Levine-Morel is weakly birationally invariant.
\end{cor}

\begin{proof} Consider the algebraic cobordism $\Omega$.
Let $f\colon Y\to X$ be a proper birational map.
The element $f_*^\Omega(1_Y)-1_X$ is in the kernel of the map
$\Omega^0(X)\to \CH^0(X)$.
By Lemma~\ref{AC:kernilp} the difference $f_*^\Omega(1_Y)-1_X$ is nilpotent,
therefore, $f_*^{\Omega}(1_Y)$ is invertible. 
The statement for an arbitrary theory 
follows by the universality of $\Omega$.
\end{proof}


\section{A sequence of split $\PP^1$-bundles} \label{OC}

In the present section we compute 
the oriented cohomology $\HH$ of a variety obtained 
as a sequence of split $\PP^1$-bundles. 
The main tool is the projective bundle theorem
for $\HH$.
Observe that all formulas are
given in terms of pull-backs and push-forwards of fundamental classes. This invariant description
will play an important role in the sequel.

\begin{lem}\label{OC:spprojbund}
Let $X$ be a smooth projective variety over a field $k$.
Let $p\colon \PP_X(\Eb) \to X$ be the projective bundle 
of a vector bundle $\Eb$ of rank 2 over $X$.
Assume that $p$ has a section $\sigma$.
Then there is a ring isomorphism
$$
\HH(\PP_X(\Eb))\simeq 
\HH(X)[\xi]/(\xi^2-y\xi),
\text{ where } \xi=\sigma_*(1_X)\text{ and }y=p^*\sigma^*\xi.
$$
\end{lem}

\begin{proof}
Consider the canonical embedding 
$\OO_{\Eb}(-1)\to p^*\Eb$, 
where $\OO_{\Eb}(-1)$ is
a tautological line bundle over $\PP_X(\Eb)$ 
(see \cite[B.5.5]{Fulton98}). 
Let $\Lb$ denote
the quotient $\Eb/\sigma^*\OO_{\Eb}(-1)$.
By Cartan formula (see \cite[Def.~3.26.(3)]{Panin03}) we have 
$$c_1^{\HH}(\Eb)=c_1^{\HH}(\sigma^*\OO_\Eb(-1))+c_1^{\HH}(\Lb) 
\qquad \text{and} \qquad 
c_2^{\HH}(\Eb)=c_1^{\HH}(\sigma^*\OO_\Eb(-1))\cdot c_1^{\HH}(\Lb).$$
According to the projective bundle theorem 
(see \cite[Thm.~3.9 and~formula~(17)]{Panin03})
applied to $p$,
there is a ring isomorphism
$$
\HH(\PP_X(\Eb))\simeq 
\HH(X)[t]/(t-a)(t-b),
$$
where $a=p^*c_1^{\HH}(\sigma^*\OO_\Eb(-1))$, 
$b=p^*c_1^{\HH}(\Lb)$ and $t=c_1^{\HH}(\OO_{\Eb}(-1))$.

Consider the elements $\xi=b \fminus t$ and $y=b \fminus a$,
where $F$ is the formal group law corresponding to $\HH$.  
Observe that by \cite[2.2.9]{Panin03b_pre} 
$\xi=c_1^{\HH}(\OO_{\Eb}(1)\otimes p^*\Lb)=\sigma_*(1_X)$
and by the very definition
$y =c_1^{\HH}(p^*\sigma^*\OO_{\Eb}(1)\otimes p^*\Lb)=p^*\sigma^*\xi$.

Since changing $t$ to $\xi$ induces an automorphism
of $h(\PP_X(\Eb))$, 
we only need to prove the relation $\xi^2-y\xi=0$. 
Note that for any power series $f$, 
we may write $f(x)-f(y)=(x-y)f'(x,y)$ 
where $f'(x,y)$ is again a power series.
Hence, we have $b\fminus t = (b\fminus t) - (t\fminus t) = 
(t-b) f'_1(b,t)$ and $(b \fminus t) - (b \fminus a)= (t-a) f'_2(a,b,t)$ 
for some power series $f'_1$ and $f'_2$. 
Therefore
$$
\xi(\xi-y)= (b \fminus t) \big((b \fminus t)-(b \fminus a)\big) 
= (t-b)(t-a)f'_1(b,t)f'_2(a,b,t) = 0$$ 
and the proof is finished.
\end{proof}

\begin{cor}\label{OC:pushfor} 
In the notation of Lemma~\ref{OC:spprojbund} we have
the following formula for the push-forward of the fundamental class
$$
p_*(1_{\PP_X(\Eb)})=\sigma^*(\xi^{-1}+(\fminus\xi)^{-1}).
$$
\end{cor}

\begin{proof} 
By \cite[Thm.~5.30]{Vishik07} we have
$$
p_*(1_{\PP_X(\Eb)})=
(\sigma^*b\fminus \sigma^*a)^{-1}+(\sigma^*a \fminus \sigma^* b)^{-1},
$$
where $\sigma^*a$ and 
$\sigma^*b$ are the roots of the bundle $\Eb$
as in the proof of Lemma~\ref{OC:spprojbund}.
Since $y=b\fminus a=p^*\sigma^*\xi$ we obtain the desired formula.
\end{proof}

\begin{thm}\label{OC:genbund} 
More generally, let $X$ be a variety obtained
by means of a sequence of split $\PP^1$-bundles, \ie
there is a sequence of varieties $X_i$, $0\le i\le N$,
starting from a point $X_0=\pt$ and finishing at $X_N=X$ 
such that for each  $1\le i\le N$
$$
p_i\colon X_i\simeq \PP_{X_{i-1}}(\Eb_i)\to X_{i-1}
$$
is a projective bundle with a section $\sigma_i$, where 
$\Eb_i$ is some vector bundle of rank 2 over $X_{i-1}$.
Then there is a ring isomorphism
$$
\HH(X_N)\simeq \HH(\pt)[\xi_1,\ldots,\xi_N]/I,
$$
where $I$ is an ideal generated by elements 
$\{\xi_i^2-y_i\xi_i\}_{i=1\ldots N}$,
$\xi_i=p^*\sigma_{i*}(1_{X_{i-1}})$, 
$y_i=p^*\sigma_i^* \sigma_{i*}(1_{X_{i-1}})$ and $p^*$ denotes
the pull-back on $X_N$.
\end{thm}

\begin{proof} Follows by induction using Lemma~\ref{OC:spprojbund}.
\end{proof}

\begin{rem} Observe that the element $y_i$ appearing in the
relations is not necessarily in $\HH(\pt)$.
To obtain a complete answer in terms
of generators $\xi_1,\ldots,\xi_N$ and relations with coefficients
in $\HH(\pt)$ one has to express the elements $y_i$
in terms of the `previous' generators $\xi_1,\ldots,\xi_{i-1}$. 
In the next section, we show how to 
obtain such an expression for Bott-Samelson
varieties using the characteristic map.
\end{rem}

\begin{ex} 
Let $X$ be as in Lemma~\ref{OC:spprojbund}.
By \cite[3.2.11]{Fulton98} we have
the following formula
for the total Chern class of 
the {\em tangent bundle} of $\PP_X(\Eb)$:
$$
c^{\HH}(\Tb_{\PP_X(\Eb)})=
c^{\HH}(p^*\Tb_X)\cdot c^{\HH}(\OO_{\Eb}(1)\otimes p^*\Eb).
$$
Since there is an exact sequence
$$
0\to \OO_{\Eb}(1)\otimes p^*\sigma^*\OO_{\Eb}(-1) \to 
\OO_{\Eb}(1)\otimes p^*\Eb \to \OO_{\Eb}(1) \otimes p^*\Lb\to 0,
$$
by Cartan formula we obtain
$$
c^{\HH}(\OO_{\Eb}(1)\otimes p^*\Eb)=
c^{\HH}(\OO_{\Eb}(1)\otimes p^*\Lb)\cdot 
c^{\HH}(\OO_{\Eb}(1)\otimes p^*\sigma^*\OO_{\Eb}(-1)).
$$
Since $\OO_{\Eb}(1)\otimes p^*\sigma^*\OO_{\Eb}(-1)=
(\OO_{\Eb}(1)\otimes p^*\Lb) \otimes 
(p^*\sigma^*\OO_{\Eb}(-1)\otimes p^*\Lb^\vee)$, we obtain that
$$
c^{\HH}(\Tb_{\PP_X(\Eb)})
=c^{\HH}(p^*\Tb_X)\cdot (1+\xi)\cdot (1+(\xi\fminus y)).
$$

More generally, 
for $X_N$ from Theorem~\ref{OC:genbund}
we obtain by induction
\begin{equation}\label{OC:gentang}
c^{\HH}(\Tb_{X_N})=\prod_{i=1}^N (1+\xi_i)(1+(\xi_i\fminus y_i)).
\end{equation}
\end{ex}


\section{An elementary step of the Bott-Samelson resolution}
We apply the results of the previous section to compute
an oriented cohomology
of a split $P/B=\PP^1$-bundle,
where $P$ is a semi-direct product
of a split reductive linear algebraic group 
of semisimple rank one and 
a connected unipotent group,
and $B$ is a Borel subgroup of $P$.
This example appears as an elementary step in the construction
of a Bott-Samelson variety.

\begin{ntt}\label{BS:defbund}
Consider a split connected solvable algebraic group $B$ over a field $k$.
Let $T$ denote its split maximal torus.
Let $X$ be a scheme on which $B$ acts on the right
such that the quotient $X/B$ exists and 
$X \to X/B$ is a principal $B$-bundle.
Let $M$ denote the group of  characters of $T$.
Each character $\lambda\in M$ extends to 
a one-dimensional representation $V_\lambda$ of $B$ and, hence,
defines a line bundle $\Lb(\lambda)$ over $X/B$ 
whose total space is the fiber product $X \times^B V_\lambda$, \ie
the quotient $(X\times_k V_\lambda)/B$ 
by means of the right $B$-action $(x,v)b=(xb,b^{-1}v)$ 
(see \cite[Ch.1~\S3]{BottSamelson}).
\end{ntt}

\begin{dfn} 
Let $\HH$ be an oriented cohomology theory and 
let $F$ be the corresponding formal group law 
over the coefficient ring $R=\HH(\pt)$.
We define a ring homomorphism $\cc\colon \RMF\to \HH(X/B)$
from the formal group ring $\RMF$
to the cohomology ring $\HH(X/B)$ 
by sending a generator $x_\lambda$
to the first Chern class $c_1^{\HH}(\Lb(\lambda))$ 
of the line bundle $\Lb(\lambda)$.
This map is well-defined since all Chern classes are nilpotent and satisfy $c_1^{\HH}(\Lb(\lambda+\mu))=c_1^{\HH}(\Lb(\lambda)\otimes \Lb(\mu))=c_1^{\HH}(\Lb(\lambda))\fplus c_1^{\HH}(\Lb(\mu))$.
It is called the {\em characteristic map}. 
\end{dfn}

\begin{ntt}\label{BS:fibprod} 
We follow the notation of \cite[\S2]{Demazure73}:
Let $P$ be a semi-direct product of a connected unipotent group $U$
and a reductive split group $L$ of semi-simple rank $1$.
Let $T$ be a maximal split torus of $L$ and 
let $M$ denote its group of characters. 
Let $\alpha$ be the one of two roots of $L$ with respect to $T$, 
let $U_\alpha$ be the corresponding unipotent subgroup
and let $B=T\cdot U_\alpha \cdot U$ be 
the Borel subgroup of $P$ containing $T$.
Let $s_\alpha(\lambda)=\lambda - \alpha^\vee (\lambda)\cdot\alpha$
denote the reflection corresponding to the root $\alpha$.

Consider the fibered product $X'=X\times^B P$, \ie
the quotient $(X\times_k P)/B$ 
by means of the $B$-action $(x,h)b=(xb,b^{-1}h)$.
By definition, 
$X'$ is a principal $P$-bundle over $X/B$,
all fibers of the canonical projection $p\colon X'/B\to X/B$ 
are isomorphic to $P/B\simeq \PP^1$ and  
there is an obvious section $\sigma\colon X/B\to X'/B$
given by $x\mapsto (x,1)$. 
According to \cite[Exrc.~7.10.(c)]{Hartshorne}
there exists a vector bundle $\Eb$ of rank 2 over $X/B$ such that 
$X'/B$ can be identified with the projective bundle 
$\PP_{X/B}(\Eb)$.
\end{ntt}

\begin{rem} 
Observe that in \cite[\S2]{Demazure73} 
instead of the quotient $X'/B$ the author considers
the quotient of $X'$ modulo the opposite Borel subgroup $B'$.
This doesn't change much, 
since there is an obvious isomorphism 
$i\colon X'/B\to X'/B'$ 
induced by an automorphism of $X'$ 
sending the class of $(x,p)$ to the class of $(x,pn_\alpha)$, 
where $n_\alpha$ is an element of $P$ 
representing the reflection $s_\alpha$.
In particular, 
for any character $\lambda$ we have 
$i^*\Lb'(\lambda)=\Lb'(s_\alpha(\lambda))$, 
where $\Lb'(\lambda)$ (resp.~$\Lb'(s_\alpha(\lambda))$)
is the line bundle over $X'/B'$ (resp.~over $X'/B$).
\end{rem}

\begin{ntt}
Applying Lemma~\ref{OC:spprojbund} to the
projective bundle $p\colon \PP_{X/B}(\Eb)\to X/B$ 
with the section $\sigma$
we obtain an isomorphism
$$
\HH(X'/B)\simeq \HH(X/B)[\xi]/(\xi^2-y\xi),
\quad\text{where }\xi=\sigma_*(1_{X/B})
\text{ and }y=p^*\sigma^*\xi.
$$
\end{ntt}

\begin{lem}\label{BS:sigpull_eq} 
We have 
$\sigma^* \xi=c_1^{\HH}(\Lb(-\alpha))=\cc(x_{-\alpha})$.
\end{lem}

\begin{proof} 
By definition of the push-forward
we have $\xi=c_1^{\HH}(\OO(D))$,
where $\OO(D)$ is a line bundle 
corresponding to the divisor $D=\sigma(X/B)$. 
The first equality then follows from \cite[\S2.5~Lem.~3]{Demazure74}.
The second one follows from the definition of the characteristic map $\cc$.
\end{proof}

\begin{cor}\label{BS:form_y} 
We have $y=p^*\sigma^*\xi=p^*\cc(x_{-\alpha})$.
\end{cor}

\begin{lem}\label{BS:charmap} 
Consider the characteristic map  $\cc'\colon\RMF \to \HH(X'/B)$.
For any $u \in \RMF$, we have
$$
(1)\;\;\sigma^*\cc'(u)=\cc(u)\quad\text{and}\quad (2)\;\; 
\cc'(u)=p^*\cc(s_{\alpha}(u)) + p^*\cc\big(\Del_{-\alpha}(u)\big)\cdot \xi
$$
\end{lem}

\begin{proof}
By \cite[\S 2.5 Lem. 1 and Prop.1]{Demazure74}, 
for any character $\lambda\in M$, we have two equalities
$$
\sigma^*\Lb'(\lambda) = \Lb(\lambda) 
\qquad \text{and} \qquad 
\Lb'(s_\alpha(\lambda))=
p^*(\Lb(\lambda)) \otimes \OO(D)^{\otimes \alpha^\vee(\lambda)},
$$ 
where $\Lb'(\lambda)$ is the corresponding line bundle over $X'/B$
and $\OO(D)$ is a 
line bundle corresponding to the divisor $D=\sigma(X/B)$.

By the first equality, formula~(1) holds for $u=x_\lambda$. 
It therefore holds for all $u \in R [M]_F$ and then for all $u \in \RMF$ by continuity. 

By the second equality, we have  for $u=x_\lambda$
$$
\cc'(x_{s_\alpha(\lambda)})=
p^*\cc(x_{\lambda})\fplus (\alpha^\vee(\lambda)) \ftimes \xi.
$$ 
Note that for any power series $g(\xi)$ in a variable $\xi$ 
we may write $g(\xi)=g(0) + \xi \cdot  g'(\xi)$, 
where $g'(\xi)=(g(\xi)-g(0))/\xi$ is again a power series.
Since, $\xi^2=y\xi$, we have $g(\xi)=g(0) + \xi \cdot g'(y)$. 
Applying this to 
$g(\xi)=p^*\cc(x_\lambda) \fplus \alpha^\vee(\lambda) \ftimes \xi$ 
and observing that $g(0)=p^*\cc(x_\lambda)$ we obtain
$$
\cc'(x_{s_{\alpha(\lambda)}})=p^*\cc(x_\lambda)+\xi \cdot g'(y),
$$
where
$g(y)=p^*\cc(x_\lambda \fplus \alpha^\vee(\lambda)\ftimes x_{-\alpha})=
p^*\cc(x_{s_\alpha(\lambda)})$
by Cor.~\ref{BS:form_y}.
Then by definition of the operator $\Del_{-\alpha}$ we obtain
$g'(y)=p^*(-\Del_{-\alpha}(x_\lambda))$ and, hence, 
$\cc'(x_{s_{\alpha(\lambda)}})=
p^*\cc(x_\lambda)-p^*\cc\big(\Del_{-\alpha}(x_\lambda)\big)\xi$. 
Since $\Del_{-\alpha} s_\alpha = -\Delta_{-\alpha}$, applying $s_\alpha$ to $u$ we obtain
$$
\cc'(x_{\lambda})=
p^*\cc(s_\alpha(x_\lambda))+
p^*\cc\big(\Del_{-\alpha}(x_\lambda)\big)\cdot \xi.
$$
For general $u$ formula~(2) then follows by induction 
on the degree of monomials and 
by continuity, using
Prop.~\ref{DO:eqDel}.\eqref{DO:DelDeriv_item} and $\xi^2 =y \xi$.
\end{proof}

\begin{rem}
Following Remark~\ref{DO:otherconv}, 
the second formula would look simpler 
with the other definition of $\Del_{\alpha}$.
\end{rem}

\begin{prop}\label{BS:PP_eq}
The following formulas hold for any $u\in \RMF$
$$
\begin{array}{lll}
(1)\quad p_*(1_{X'/B})=\cc\big(\CC_{\alpha}(1)\big), & &
(3)\quad p^*\big(\cc\big(\CC_{\alpha}(u)\big)\big)= 
\cc'\big(\CC_{\alpha}(u)\big),\\
(2)\quad p_*\big(\cc'(u)\big)=\cc\big(\CC_{\alpha}(u)\big), & &
(4)\quad p^* p_* (\cc'(u)) = \cc'(\CC_{\alpha}(u)).
\end{array}
$$
\end{prop}

\begin{proof}
The first formula of the proposition follows 
from Corollary~\ref{OC:pushfor} 
applied to the projective bundle $p\colon X'/B\to X/B$ 
where we identify
$\sigma^*\xi$ with $\cc(x_{-\alpha})$ 
according to Lemma~\ref{BS:sigpull_eq}.

By point (2) of Lemma~\ref{BS:charmap} 
and the projection formula
we have
$$
p_*\big(\cc'(u)\big) = 
\cc(s_\alpha(u))\cdot p_*(1_{X'/B}) +
\cc\big(\Del_{-\alpha}(u)\big)\cdot p_*(\xi) =
$$
By the first formula and the fact that 
$p_*(\xi)=p_*(\sigma_*(1_{X/B}))=1_{X/B}$ 
it can be rewritten as
$$
\cc(s_\alpha(u))\cdot \cc\big(\CC_\alpha(1)\big) +
\cc\big(\Del_{-\alpha}(u)\big)=
\cc\big(s_\alpha(u) \CC_{\alpha}(1) + \Del_{-\alpha}(u)\big)
$$
Using the definition of $C_\alpha$ and point \eqref{DO:eqCCCs} of
Proposition~\ref{DO:eqCC}
we finally obtain
$$
=\cc(s_\alpha(u)\cdot \ee_\alpha + \Del_{-\alpha}(u))=
\cc(C_{-\alpha}(s_{-\alpha}(u)))=\cc(C_\alpha(u)).
$$
This proves the second formula.

The third formula follows from point (2) of Lemma~\ref{BS:charmap}
and the fact that $\CC_{-\alpha}\Del_{-\alpha}=0$ 
(see (6) of Prop.~\ref{DO:eqCC}).
The last formula is obtained by the composite of 
the second and the third one.
\end{proof}


\section{Bott-Samelson varieties}

In the present section we compute
an oriented cohomology of Bott-Samelson varieties.
For definition and basic properties of Bott-Samelson
varieties, we refer to papers
\cite{BottSamelson}, \cite{SpringerNotes} and \cite{Bressler92}.

\begin{ntt}
Let $G$ be a split semisimple linear algebraic 
group over a field $k$,
let $T$ be a split maximal torus of $G$ and let $T\subset B$
be a Borel subgroup.
Let $\{\alpha_1, \alpha_2,\ldots, \alpha_n\}$, 
where $n$ is the rank of $G$, 
be a set of simple roots of $G$.
Let $P_i$
be a minimal parabolic subgroup corresponding
to a simple root $\alpha_i$, \ie the subgroup generated by unipotent
subgroups $U_{\alpha_j}$, $1\le j\le n$ and $U_{-\alpha_i}$.
\end{ntt}

\begin{dfn}
For an $l$-tuple of integers 
$I=(i_1,i_2,\ldots,i_l)$ with $1\le i_j\le n$,
we define a variety $X_I$ to be the fiber product
$$
X_I=P_{i_1}\times^B P_{i_2}\times^B\ldots \times^B P_{i_l}.
$$
If $I=\emptyset$, then we set $X_\emptyset=pt$.

Observe that there is a natural right action of $B$ on $X_I$,
the quotient $X_I/B$ exists 
and $X_I\to X_I/B$ is a principal
$B$-bundle. The variety $X_I/B$ is called a Bott-Samelson variety
corresponding to $I$.
\end{dfn}

\begin{ntt}
Let $J=(i_1,i_2,\ldots, i_{l-1})$.
Then $X_I=X_J\times^B P_{i_l}$, and we are in the situation
of \ref{BS:fibprod} for $P=P_{i_l}$, $\alpha=\alpha_{i_l}$, 
$X'=X_I$ and $X=X_J$.  
In particular, 
the projection map 
$$
p_I\colon X_I/B\to X_J/B
$$ 
has a structure
of a $\PP^1$-bundle with a section denoted by $\sigma_I$.

Observe that the structure map $\pi_I\colon X_I/B\to \pt$ 
can be written as a composite of projection maps
$$
X_I/B\to X_{(i_1,\ldots,i_{l-1})}/B\to \ldots \to X_{(i_1)}/B\to \pt.
$$
Hence, the variety $X_I/B$ is obtained by means of a sequence
of split $\PP^1$-bundles.

According to \ref{BS:defbund}
for an oriented cohomology theory $\HH$ 
with a formal group law $F$ there is a well-defined
characteristic map 
$$
\cc_I\colon \RMF \to \HH(X_I/B), 
$$
where $\RMF$ is the formal group ring of the group of characters 
$M$ of $T$. Observe that $\cc_{\emptyset}=\aug$ is the augmentation map.
Applying Theorem~\ref{OC:genbund} and Corollary~\ref{BS:form_y} 
we obtain an isomorphism
\begin{equation}\label{BS:bottiso}
\HH(X_I/B)\simeq \HH(\pt)[\xi_1,\xi_2,\ldots,\xi_l]/
(\{\xi_j^2-y_j\xi_j\}_{j=1\ldots l}),
\end{equation}
where $y_j=p^*\cc_{(i_1,\ldots,i_{j-1})}(x_{-\alpha_{i_j}})$ 
and $p^*$ denotes the pull-back to $\HH(X_I/B)$.
\end{ntt}

\begin{thm} \label{BS:xiTheta_thm} 
Let $X_I/B$ be the 
Bott-Samelson variety corresponding to 
an $l$-tuple $I=(i_1,\ldots,i_l)$.
For any subset $K$ of $[1,l]$ we define
$$
\xi_K=\prod_{j\in K} \xi_j\quad\text{ and }\quad
\Theta_K=\Theta_1\cdots \Theta_l,\;\text{ where }
\Theta_j=
\begin{cases}
\Delta_{-\alpha_{i_j}} & \text{if } j\in K, \\
s_{\alpha_{i_j}} & \text{otherwise.}
\end{cases}
$$

Then the elements $\xi_K$, where $K$ runs through all subsets of
$[1,l]$, form a basis of the free $\HH(\pt)$-module 
$\HH(X_I/B)$. As a ring $\HH(X_I/B)$ is a quotient of
the polynomial ring $\HH(\pt)[\xi_1,\ldots,\xi_l]$ modulo
the relations
$$
\xi_j^2=\sum_{K\subseteq [1,j-1]}\aug \Theta_K(x_{-\alpha_{i_j}})\xi_K\xi_j.
$$
\end{thm}

\begin{proof} The fact that $\HH(X_I/B)$ is a free $\HH(\pt)$-module
with a basis $\xi_K$ follows by the 
projective bundle theorem, since $X_I/B$ 
is obtained by means of a sequence of split $\PP^1$-bundles. 
To obtain the relations, we first use 
Lemma~\ref{BS:charmap}.(2) to prove by induction that
\begin{equation} \label{ccI_eq}
\cc_{(i_1,\ldots,i_j)}(u)=\sum_{K\subseteq [1,j]}\aug\Theta_K(u)\xi_K.
\end{equation}
Using this we then compute $y_j$ and 
plug the result into the formula~\eqref{BS:bottiso}. 
\end{proof}

The following lemma provides a formula for the push-forward
of the structure map $\pi_I\colon X_I/B\to \pt$:

\begin{lem} \label{BS:pushcc}
Consider the variety $X_I/B$ for a tuple $I=(i_1,\ldots,i_l)$.
For any $u\in\RMF$ we have
$$
\pi_{I*}\big(\cc_I(u)\big)=
\aug \CC_I(u),
$$
where $\CC_I=\CC_{\alpha_1}\ldots \CC_{\alpha_l}$ is the composite
of operators defined in \ref{DO:defC}.
\end{lem}

\begin{proof}
Decomposing the structure map as
$\pi_I\colon X_I/B \stackrel{p_I}\to X_J/B \stackrel{\pi_J}\to \pt$,
where $J=(i_1,\ldots, i_{l-1})$ and applying
Prop.~\ref{BS:PP_eq}.(2) we obtain
$$
\pi_{I*}\big(\cc_I(u)\big)=\pi_{J*}p_{I*}\big(\cc_I(u)\big)=
\pi_{J*}\big(\cc_J(C_{\alpha_l}(u))\big).
$$
Repeating this recursively for the structure map $\pi_J\colon X_J/B\to \pt$ 
and so on, we obtain the desired formula.
\end{proof}


\section{Bott-Samelson resolutions} \label{BR}

In the present section we discuss the relations between  
Bott-Samelson varieties and Schubert varieties on $G/B$. 

\begin{ntt} 
According to Bruhat decomposition the variety of complete flags 
$G/B$ is a finite 
disjoint union of affine spaces  
$G/B=\coprod_{w\in W} BwB/B$, where $W$ is the Weyl group of $G$.
The closure of a cell $BwB/B$ is denoted by $X_w$
and is called a Schubert variety. Observe that $X_w$ is not smooth
in general.

Let $w=s_{i_1}s_{i_2}\ldots s_{i_l}$ be a (reduced)
decomposition into a product
of simple reflections (here $s_i$ denotes $s_{\alpha_i}$) with $l=l(w)$
being the length of $w$. The multiplication map induces a morphism
$$
q_I\colon X_I/B\to G/B,\text{ where }I=(i_1,i_2,\ldots,i_l).
$$
By the results of \cite[Ch.~3]{BottSamelson} 
(see also \cite[\S1.7]{SpringerNotes} and \cite[\S4]{Bressler92}) 
this map factors as
$q_I\colon X_I/B\to X_w\hookrightarrow G/B$, 
where the first map is surjective birational
and the second is a closed embedding. Furthermore, when $J$ is $I$ with the last entry removed, we have $q_I \circ \sigma_I=q_J$.
\end{ntt}

\begin{lem}\label{BR:bcha} 
Consider the characteristic map $\cc_{G/B}\colon \RMF\to \HH(G/B)$.
For any $u\in\RMF$ we have
$$
q^*_I(\cc_{G/B}(u))=\cc_I(u).
$$
\end{lem}

\begin{proof} 
By definition of the map $q_I$ and the bundle $\Lb(\lambda)$ we have $q_I^*(\Lb(\lambda))=\Lb_I(\lambda)$, where $\Lb(\lambda)$ (resp.~$\Lb_I(\lambda)$) is the line bundle over $G/B$ (resp.~over $X_I/B$) corresponding to a character $\lambda$. 
Hence, $q_I^*(\cc_{G/B}(x_\lambda))=\cc_I(x_\lambda)$. 
For a general $u$ it follows by continuity.
\end{proof}

Let $w_0$ be the element of maximal length in $W$ and let $X=X_{I_0}/B$ be the Bott-Samelson variety where $I_0$ is a reduced decomposition of $w_0$. 

\begin{lem}\label{BR:injpback}  
Assume that $\HH$ 
is weakly birationally invariant.
Then the pull-back $q_{I_0}^*$ is injective.
\end{lem}

\begin{proof}
Since our theory is weakly birationally invariant, 
the push-forward of the fundamental class
$q_{I_0*}(1)$ is invertible. The lemma now follows by projection formula.
\end{proof}

\begin{thm}[\cf {\cite[Prop. 3]{Bressler92}}]\label{BR:pipi_eq} 
Assume that $\HH$ is weakly birationally invariant.
Consider the quotient map $p_i\colon G/B\to G/P_i$, 
where $1 \leq i \leq n$. Then for any $u\in\RMF$ 
we have the equality
$$
p_i^* p_{i*} \big(\cc_{G/B}(u)\big)=
\cc_{G/B}\big(\CC_{\alpha_i}(u)\big)
$$
\end{thm}

\begin{proof}
We choose a reduced decomposition $I_0$ of $w_0$ in simple reflections 
such that the last reflection is $s_i$. 
Let $J$ be the same tuple without the last entry.
Consider the Cartesian diagram
\begin{equation}\label{BR:diagr}
\xymatrix{
X_{I_0}/B\ar[r]^{p_{I_0}}\ar[d]_{q_{I_0}}& X_J/B \ar[d]^{p_i \circ q_J}\\
G/B\ar[r]^{p_{i}}&G/P_i
}
\end{equation}
Since the horizontal map $p_i$ is smooth,
the diagram is transversal.
Since $q_{I_0}^*$ is injective by Lemma~\ref{BR:injpback}, 
it is enough to prove the equality 
after applying $q_{I_0}^*$ to both sides. 
By base change and by transversality of the diagram
we have
$$
q_{I_0}^* p_i^* p_{i*}\big(\cc_{G/B}(u)\big) 
=p_{I_0}^*(p_i \circ q_J)^*p_{i*}\big(\cc_{G/B}(u)\big)
=p_{I_0}^* p_{I_0*} q_{I_0}^*\big(\cc_{G/B}(u)\big)
$$ 

By Lemma~\ref{BR:bcha}
and by Prop.~\ref{BS:PP_eq}.(4) we obtain
$$
=p_{I_0}^* p_{I_0*} \big(\cc_{I_0}(u)\big) 
=\cc_{I_0}\big(\CC_{\alpha_i}(u)\big)
=q_{I_0}^*\big(\cc_{G/B}\big(\CC_{\alpha_i}(u)\big)\big)
$$
and the theorem is proved.
\end{proof}


\part{Algebraic and geometric comparison. Applications.}

\section{Comparison results}\label{AG}

In the present section we explain why the ring $\HMF$ is naturally isomorphic to the ring $\HH^*(G/B)$ (see \ref{AG:isoHMFHGB_thm}) for most weakly birationally invariant theories. 

\begin{ntt}\label{AG:basis}
Let $I$ be a sequence of simple reflections. Recall that $w(I)=s_{i_1}\cdots s_{i_l}$ is the corresponding product of simple reflections. Let
$$
\bb_I=(q_{I})_*(1) \in \HH^*(G/B)
$$ 
denote the push-forward of the fundamental class of the Bott-Samelson variety corresponding to $I$.
\end{ntt}

\begin{ass} \label{AG:basis_ass}
For each element $w\in W$, let $I_w$ be a chosen reduced decomposition of $w$. The elements $\bb_{I_w}$, 
where $w$ runs through all elements of the Weyl group $W$,
form an $R$-basis of the cohomology $\HH^*(G/B)$.
\end{ass}

\begin{lem} \label{AG:basiscobord_lem}
Assumption \ref{AG:basis_ass} holds for
\begin{enumerate}
\item \label{AG:specialization_item} any oriented cohomology theory $\HH_1^*=\HH_0^*\otimes_{R_0} R_1$ obtained from another one $\HH_0^*$ for which the assumption already holds;
\item \label{AG:Chow_item} Chow groups and Chow groups mod $n$ (in arbitrary characteristic);
\item \label{AG:cobordism_item} algebraic cobordism, and therefore all oriented cohomology theories over a base field of characteristic zero;
\item \label{AG:K_item} $K$-theory (and $K$-theory mod $n$) in arbitrary characteristic. 
\end{enumerate}
\end{lem}
\begin{proof}
Point \eqref{AG:specialization_item} is obvious since a free module stays free by base change. For Chow groups and K-theory, this is classical and for example proved in \cite{Demazure74}. For algebraic cobordism, since $G/B$ is a cellular space and classes of Schubert varieties form a $\ZZ$-basis of $\CH(G/B)$, the lemma follows from \cite[Cor.~2.9]{VishikY}. 
\end{proof}

We consider the characteristic map $\cc_{G/B}\colon \RMF \to \HH^*(G/B)$.

\begin{lem} \label{AG:piZ_lem}
Let $\pi\colon G/B\to \pt$ be the structure map. 
For any sequence $I$ and $u\in \RMF$ we have
$$\pi_*(\bb_I\cdot \cc_{G/B}(u))=\aug\CC_I(u).$$
\end{lem}
\begin{proof}
The projection formula and Lemma~\ref{BS:pushcc} imply that we have
$$\pi_*(\bb_I\cdot \cc_{G/B}(u))=\pi_*(q_{I*}(1)\cdot \cc_{G/B}(u))=
\pi_* q_{I*}q_I^*(\cc_{G/B}(u))=\pi_{I*}(\cc_I(u))=\aug\CC_I(u).$$
\end{proof}

Recall from Definition \ref{DR:charmap_def} that we also have an algebraic characteristic map $\cc: \RMF \to \HMF$.

\begin{lem} \label{AG:thetatorinv_lem}
Assume the torsion index $\tor$ is invertible in $R$ and that Assumption \ref{AG:basis_ass} is satisfied. Then, there is a unique morphism $\theta: \HMF \to \HH^*(G/B)$ of $R$-algebras such that $\cc_{G/B}=\theta \circ \cc$. 
\end{lem}
\begin{proof}
We fix a reduced decomposition $I_w$ for each $w \in W$.
By \ref{DR:surjchar}, the map $\cc$ is surjective. 
We can therefore find a pre-image $u_w$ of each element $\zz_{I_w}^\CC$ of Theorem \ref{DR:charMapBasis}. 
By the form of the characteristic map given there, these elements satisfy $\aug \CC_{I_w}(u_{v})=\delta_{w,v}$. 
Let us set $\ab_{I_w}=\cc_{G/B}(u_w)$. 
By the previous lemma, $(\ab_{I_w})_{w \in W}$ is a basis of $\HH^*(G/B)$ that is dual basis to the basis $(\bb_{I_w})_{w \in W}$ with respect to the bilinear form $(\bb,\ab)\mapsto \pi_*(\bb.\ab)$, and this form is therefore nondegenerate. 
We now define $\theta$ as the $R$-linear morphism sending $\zz_{I_w}^\CC$ to $\ab_{I_w}$. By the formula $\pi_*(\bb_{I_w}.\cc_{G/B}(u))=\aug \CC_{I_w}(u)$, we must have 
\begin{equation} \label{AG:charmap_eq}
\cc_{G/B}(u)=\sum_{w \in W} \aug \CC_{I_w}(u) \ab_{I_w}
\end{equation} 
and the isomorphism $\theta$ therefore satisfies $\cc_{G/B}= \theta \circ \cc$ since the $\zz_{I_w}^{\CC}$ satisfy the same formula with $\cc$.
The uniqueness of $\theta$ is immediate by surjectivity of $\cc$.
\end{proof}

\begin{cor}
When $\tor$ is invertible in $R$ and Assumption \ref{AG:basis_ass} holds, the characteristic map $\cc_{G/B}$ is surjective.
\end{cor}
\begin{proof}
Use Theorem \ref{DR:surjchar} and the previous lemma. 
\end{proof}

\begin{cor}
When $\tor$ is invertible in $R$ which has no $2$-torsion and Assumption \ref{AG:basis_ass} holds, the kernel of the characteristic map $\cc_{G/B}$ is the ideal of $\RMF$ generated by elements in $\mI$ fixed by $W$. 
\end{cor}
\begin{proof}
Use Theorem \ref{DR:kercharmap_thm} and the previous lemma.
\end{proof}

Let $u_0 \in \mI^N$ be chosen as in section \ref{IT}.

\begin{lem} \label{AG:bIwcGB_lem}
Let $I_w$ be a reduced decomposition of an element $w$. Then
$$\bb_{I_w}\cdot\cc_{G/B}(u_0)=\left\{\begin{array}{ll}
0 & \text{ if } w \neq w_0 \\
\cc_{G/B}(u_0) & \text{ if } w=w_0
\end{array}
\right.$$
\end{lem}
\begin{proof}
We have $\bb_{I_w}\cdot \cc_{G/B}(u_0)=q_{I_w *}(\cc_{I_w}(u_0))$
by the projection formula and the fact that $q_{I_w}^*\cc_{G/B}=\cc_{I_w}$. 
By formula \eqref{ccI_eq}, $\cc_{I_w}(u_0)=0$ if $l(I_w)<N$ \ie\ if $w\neq w_0$. 
This proves the first part of the formula. 
The other case will follow if we show that $\bb_{I_{w_0}}=1-\sum_{w \neq w_0}r_w \bb_{I_w}$ for some coefficients $r_w$ in $R$. 
For this, it suffices to show that the coefficient $r_{w_0}$ in front of $\bb_{I_{w_0}}$ in the decomposition of $1$ on the basis of the $\bb_{I_w}$ is $1$. We have the commutative diagram
$$\xymatrix{
U \ar@{}[dr]|{\Box} \ar[r]^{i_{U'}} \ar@{=}[d] & X_{I_0}/B \ar[d]^{q_{I_0}} & X_{I'}/B \ar[l]_{\sigma} \ar[dl]^{q_{I'}} \\
U \ar[r]^{i_{U}} \ar[r] & G/B
}$$
where $I_0$ is any reduced decomposition of $w_0$, $I'$ is $I_0$ with its last entry removed, $\sigma$ is the section of $p_{I_0}$, $U$ is the big open cell in $G/B$ and $U'$ is its pre-image in $X_{I_0}/B$. 
By \cite[\S 3.10 and 3.11]{Demazure74}, $U'$ is indeed isomorphic to $U$ and it is included in the complement of $\sigma{X_{I'}/B}$. 
Thus, $i_{U'}^* \sigma_*=0$, using that $X_{I_0}/B \to X_{I'}/B$ is a projective bundle. 
By base change on the cartesian diagram, we obtain $i_{U}^* q_{I' *}=0$ and $i_U^* q_{I_0 *}(1)=1$. Since any $I_w$ can be completed to a decomposition of $w_0$, this proves that $i_U^* (\bb_{I_w})=0$ if $w \neq w_0$.
Pulling back $1=\sum_w r_w \bb_{I_w}$ by $i_U$, we get $1=r_{w_0}$ (by homotopy invariance, $\HH^*(U) \simeq R$).
\end{proof}

\begin{lem} \label{AG:ptcu0_lem}
We have 
$$\tor\;\bb_{\emptyset}=\cc_{G/B}(u_0) \qquad\text{and}\qquad \tor\;\bb_I=\cc_{G/B}\big(\CC_{I^{\rev}}(u_0)\big).$$
\end{lem}
\begin{proof}
Let $I_{w_0}=(i_1,\ldots,i_N)$.
With the notation of Theorem \ref{BS:xiTheta_thm}, we have 
$$\tor\; \bb_{\emptyset}=\tor\;q_{\emptyset *}(1)=\tor\; q_{I_{w_0}*}\sigma_{I_{w_0}*} \cdots \sigma_{(i_1) *}(1)=q_{I_{w_0}*}(\tor\; \xi_{[1,N]})=q_{I_{w_0}*}(\cc_{I_{w_0}}(u_0))$$
$$=q_{I_{w_0}*}(1) \cc_{G/B}(u_0)=\cc_{G/B}(u_0)$$
where the two last equalities follow from the projection formula together with $q_{I_w}^*\cc_{G/B}=\cc_{I_w}$ and then Lemma \ref{AG:bIwcGB_lem}. 
For the second formula, which coincides with the first one when $I$ is empty, we use induction on the length of $I$. 
If $I=(i_1,\ldots,i_l)$ and $J=(i_1,\ldots,i_{l-1})$, then
by base change via diagram~\eqref{BR:diagr} we obtain 
$$
\tor\cdot \bb_I=\tor\;\cdot q_{I*}(1)=\tor\cdot q_{I*}p_I^*(1)=
\tor\cdot p_{i_l}^*(p_{i_l}\circ q_{J})_*(1)=
p_{i_l}^*p_{i_l*}\big(\tor\cdot \bb_{J}\big),
$$
By the induction step and Theorem~\ref{BR:pipi_eq} we obtain
$$
=p_{i_l}^*p_{i_l*}(\cc_{G/B}(\CC_{J^{\rev}}(u_0)))
=\cc_{G/B}(\CC_{\alpha_{i_l}}\circ\CC_{J^{\rev}}(u_0))
=\cc_{G/B}(\CC_{I^{\rev}}(u_0))
$$
and the proof is finished.
\end{proof}

\begin{lem}
Under the assumptions of Proposition \ref{AG:thetatorinv_lem}, we have $\bb_{\emptyset} = \ab_{I_{w_0}}$. In particular, $\ab_{I_{w_0}}$ does not depend on the choices of the reduced decompositions $I_w$.
\end{lem}
\begin{proof}
By definition of the dual basis $(\ab_{I_w})_{w \in W}$, it suffices to show that we have $\pi_*(\bb_{I_w}\cdot \bb_{\emptyset})=\delta_{w_0,w}$. 
By lemmas \ref{AG:ptcu0_lem} and \ref{AG:piZ_lem}, we have 
$$\pi_*(\bb_{I_w}\cdot \bb_{\emptyset})=\pi_*(\bb_{I_w}\cdot \tor^{-1}\cc_{G/B}(u_0)) = \tor^{-1}\aug \CC_{I_w}(u_0)= \delta_{w,w_0}.\qedhere$$ 
\end{proof}

Let $\AaGB_i$ denote the operator $p_i^* (p_i)_*$ on $\HH^*(G/B)$ and let $\AaGB_I$ be as usual for a sequence $I$. 

\begin{lem} \label{AG:thetaA_lem}
The isomorphism $\theta$ of Lemma \ref{AG:thetatorinv_lem} satisfies $\theta \Aa_I=\AaGB_I \theta$ and therefore sends the $\Aa_{I_w}(\zz_0)$ to the $\AaGB_{I_w}(\bb_\emptyset)$. 
\end{lem}
\begin{proof}
It follows from the surjectivity of the characteristic map, the fact that $\theta \circ \cc = \cc_{G/B}$, Theorem \ref{BR:pipi_eq} and Proposition \ref{PS:AandB_prop} point \eqref{PS:AcBc_item}.
\end{proof}

We can now show that the isomorphism $\theta$ of Lemma \ref{AG:thetatorinv_lem} descends to the case where $\tor$ is regular but not necessarily invertible.

\begin{thm} \label{AG:isoHMFHGB_thm}
Let $\HH^*$ be a weakly birationally invariant theory satisfying Assumption \ref{AG:basis_ass} and such that $\tor$ is regular in $R=\HH^*(\pt)$. Then there is a unique isomorphism of $R$-algebras $\theta:\HMF \to \HH^*(G/B)$ such that the characteristic map $\cc_{G/B}$ is given by $\theta\circ \cc$. It satisfies $\theta \circ \Aa_I = \AaGB_I \circ \theta$. 
\end{thm}
\begin{proof}
We know from \ref{PS:AIbasisHMF_prop} that the $\Aa_{I_w}(\zz_0)$ are a basis of $\HMF$ and from Assumption \ref{AG:basis_ass} that the $\bb_{I_w}$ are a basis of $\HH^*(G/B)$. We define $\theta$ as the isomorphism sending $\Aa_{I_w}(\zz_0)$ to $\bb_{I_w}$. By Lemma \ref{AG:thetaA_lem}, this definition coincides with the one of Lemma \ref{AG:thetatorinv_lem} when $\tor$ is invertible in $R$. We therefore have a commutative diagram$$\xymatrix{
\HMF \otimes_R R[\tor^{-1}] \ar[r]^{\theta \otimes id} & \HH^*(G/B)\otimes_R R[\tor^{-1}] \\
\HMF \ar[u] \ar[r]^{\theta} & \HH^*(G/B) \ar[u]
}$$ 
where $\theta \otimes id$ is the $\theta$ of Lemma \ref{AG:thetatorinv_lem} and the vertical maps are injective since $\HMF$ and $\HH^*(G/B)$ are free $R$-modules and $R$ injects in $R[\tor^{-1}]$. Since all other maps are ring morphisms, $\theta$ is one too, and the required equalities follow from the same ones in the case where $\tor$ is invertible.
The morphism $\theta$ is unique since it is unique after inverting $\tor$.
\end{proof}

\begin{rem}
Note that the morphism $\theta$ is independent of the choices of the $I_w$ used in its construction. This is obvious by surjectivity of the characteristic map after inverting $\tor$.
\end{rem}

This completes the identification of $\HMF$, the algebraic model for the cohomology ring introduced in section \ref{IT}, with the actual cohomology $\HH^*(G/B)$. 


\section{Formulas for push-forwards} \label{MA}

In the present section we assume that $\tor$ and $2$ are regular in $R$. Note that given a formal group law $F$, both $\RMF$ and $\HMF$ can be defined without any reference to a cohomology theory. The goal of this section is to define an ``algebraic push-forward'' $\pr\colon \HMF\to R$ and to prove formulas related to it. Of course, when the formal group law comes from a cohomology theory, we prove that the morphism $\pr$ corresponds to its geometric counterpart $\pi_*\colon \HH(G/B) \to R$ through the isomorphism $\theta\colon \HMF \stackrel{\simeq} \to \HH^*(G/B)$ from Theorem \ref{AG:isoHMFHGB_thm}. The general idea of the proofs is to use that there is a universal formal group law $U$ over the Lazard ring $\LL$, and therefore that such formulas can be proved in $\HFG{\LL}{M}{U}$, where they hold for geometric reasons, using algebraic cobordism. Then, they hold in any $\HMF$ by specialization (even if it has no geometric origin).

\begin{ntt} Recall that since $(U,\LL)$ is the universal formal group law, for any given formal group law $F$ over a ring $R$, there is a unique morphism $f: \LL \to R$ sending $U$ to $F$. By Lemma \ref{FGR:funct} and Propositions \ref{fHfirst_prop} and \ref{fHringmor_prop}, there is a commutative diagram
$$\xymatrix{
\FGR{\LL}{M}{U} \ar[d]^{f_*} \ar[r]^{\cc^U}& \HFG{\LL}{M}{U} \ar[d]^{f_\mH} \\
\FGR{R}{M}{F} \ar[r]^{\cc^F} & \HFG{R}{M}{F} 
}$$
Furthermore, the morphism $f_\mH$ commutes with the operators $\Aa_I$ and $\Bb_I$ and sends $z_0^U$ to $z_0^R$ by definition. Therefore, 
it sends $\Aa_I^U(z_0^U)$ to $\Aa_I^F(z_0^R)$ for any sequence $I$.

Let $(I_w)_{w \in W}$ be a choice of reduced decompositions. 
Then, by Prop.~\ref{PS:AIbasisHMF_prop}, the elements 
$(\Aa_{I_w}(z_0))_{w \in W}$ form an $R$-basis of $\HMF$. 
\end{ntt}

\begin{dfn} \label{algebraicPushForward_dfn}
Let $\pr: \HMF \to R$ be the $R$-linear morphism defined by $\pr(\Aa_{I_w}(z_0))=\aug \CC_{I_w}(1)$. 
\end{dfn}

\begin{prop} \label{prformulas_prop}
We have
\begin{enumerate}
\item\label{prfH_item} for any morphism $f:R \to R'$ of rings sending a formal group law $F$ to $F'$, we have $\pr^F = \pr^{F'} \circ f_\mH$.

\item \label{prreplacement_item} Assume that $\HH$ is an oriented cohomology theory satisfying the assumptions of Theorem \ref{AG:isoHMFHGB_thm}, in which the isomorphism $\theta\colon \HMF \stackrel{\simeq} \to \HH^*(G/B)$ is defined. Then the morphism $\pr$ satisfies
$\pr \circ \theta = \theta \circ \pi_*$.
In other words, the morphism $\pr$ is an algebraic replacement for the push-forward $\pi_*$. 
\end{enumerate}
\end{prop}
\begin{proof}
Since $f_\mH(\Aa^F_{I_w}(z_0^F)) = \Aa^{F'}_{I_w}(z_0^{F'})$, 
\eqref{prfH_item} follows from the definition of $\pr$.
Since the basis of $\HH(G/B)$ formed by the $\bb_{I_w}$ and the basis 
of $\HMF$ formed by the $\Aa_{I_w}(z_0)$ corresponds to each other through the isomorphism
$\theta$, \eqref{prreplacement_item} follows from Lemma~\ref{AG:piZ_lem}. 
\end{proof}

\begin{cor}\label{prformula_eq}
For any sequence $I$ the morphism $\pr$ of the proposition satisfies
$$\pr\big(\Aa_{I^{\rev}}(z_0)\cc(u)\big) = \aug \CC_I(u).$$
In particular, it is independent of the choice $(I_w)_{w \in W}$ of reduced decompositions. 
\end{cor}
\begin{proof}
By Proposition \ref{prformulas_prop}.\eqref{prreplacement_item} and by 
Lemma~\ref{AG:piZ_lem}, the formula holds for $\HFG{\LL}{M}{U}$. In general, it reduces to the case where $\tor$ is invertible. When  $v \in \RMF^W$, we have $\aug \CC_I(vu) = \aug(v)\aug\CC_I(u)$ and $\cc(vu)=\epsilon(v)\cc(u)$ so by decomposing $u$ on the $\CC_{I_w}(u_0)$, which form an $\RMF^W$-basis of $\RMF$, it suffices to prove the formula for $u=\CC_{I_w}(u_0)$. For those, it follows by specialization using $f_\mH$ from $\HFG{\LL}{M}{U}$.
\end{proof}

\begin{rem}
Note that the elements $\pr(\Aa_{I^{\rev}}(\zz_0))$ are particularly important because they represent the images of desingularized Schubert varieties in the cohomology of the point.
\end{rem}

\begin{prop}
The morphism $\pr$ satisfies
$$\pr\left(\Aa_{I^{\rev}_v}(\zz_0).\zz^\CC_{I_w}\right)=\delta_{v,w}.$$
In other words, the bilinear form $\pr(-,-)$ is nondegenerate and the bases $\left(\zz^\CC_{I_w}\right)_{w \in W}$ and $\left(\Aa_{I^{\rev}_w}(\zz_0)\right)_{w \in W}$ are dual to each other. 
\end{prop}
\begin{proof}
The formula can be computed after extending scalars to $R\left[\tor^{-1}\right]$. Then, since the characteristic map $\cc$ is surjective, the elements $\zz^\CC_{I_w}$ have preimages $u_w$, such that $\aug \CC_{I_v}(u_w)=\delta_{v,w}$, by the expression of $\cc$ given in Theorem \ref{DR:charMapBasis}. The result then follows from formula \eqref{prformula_eq}.
\end{proof}

Let $I_0$ be a reduced decomposition of the longest element $w_0$. When $\tor$ is invertible, the element $y_0=\CC_{I^{\rev}_0}(u_0)/\tor$ is invertible (by Lemma \ref{DR:reducedI0}), and $\cc(y_0)=\Aa_{I^{\rev}_0}(\zz_0)$ by Proposition \ref{PS:AandB_prop}. 
\begin{lem}
When $\tor$ is invertible in $R$, the morphism $\pr$ satisfies
$$\pr(\cc(u))=\aug\CC_{I_0}\left(u y_0^{-1}\right).$$ 
\end{lem}
\begin{proof}
We have
$$\pr(\cc(u))=\pr\big(\cc(uy_0^{-1})\cc(y_0)\big)=\pr\big(\cc(uy_0^{-1}\Aa_{I^{\rev}_0}(u_0))\big)=\aug\CC_{I_0}(uy_0^{-1}).\qedhere$$
\end{proof}

\begin{prop}
The operators $\CC$ satisfy the formula
$$\aug \CC_I(u\CC_J(u_0))=\aug \CC_{J^{\rev}}(u\CC_{I^{\rev}}(u_0))$$
for any pair of sequences $I$ and $J$. In particular ($u=1$, $J=\emptyset$), for any sequence $I$, we have $\aug \CC_I(u_0) = \aug \CC_{I^{\rev}}(u_0)$.
\end{prop}
\begin{proof}
The formula can be proved after extension to $R\left[\tor^{-1}\right]$. It then follows from the computation of $\pr$ in formula \eqref{prformula_eq} applied to both ends of
$$\Aa_{I^{\rev}}(\zz_0)\cc\big(u \CC_J(u_0)\big)=\tor \cc\big(\CC_{I^{\rev}}(u_0) u \CC_J(u_0)\big)= \cc\big(u \CC_{I^{\rev}}(u_0)\big)\Aa_J(\zz_0).\qedhere$$
\end{proof}

\begin{rem}
The formulas given in this section are purely algebraic, but their proofs use geometric properties of cobordism. It would be interesting to have algebraic proofs derived directly from the various formulas in section \ref{DO}.  
\end{rem}


\section{Algorithm for multiplying in $\HH^*(G/B)$}\label{AM}

Fixing a choice of reduced decompositions $I_w$ for all $w \in W$, we define $\ab_{I_w}$ as $\theta(\zz^\CC_{I_w})$.

\begin{prop}
Under the assumptions of Theorem \ref{AG:isoHMFHGB_thm}, the bilinear form $(b,a) \mapsto \pi_*(b\cdot a)$ is non degenerate on $\HH^*(G/B)$ and for any choice of reduced decompositions $I_w$, the basis $(\bb_{I_w})$ and $(\ab_{I_w})$ are dual to each other.
\end{prop}
\begin{proof}
The form can be computed after scalar extension to $R[\tor^{-1}]$, in which case it follows from the proof of Lemma \ref{AG:thetatorinv_lem}.
\end{proof}

Now the multiplication algorithm goes as follows: 
Substituting $u=\CC_{I_w^{\rev}}(u_0)$ in \eqref{AG:charmap_eq} 
and using that $\tor\; \bb_{I_w}=\cc_{G/B}\big(\CC_{I_w^{\rev}}(u_0)\big)$ we obtain the transition
matrix from the basis $(\bb_{I_w})$ to the basis $(\ab_{I_w})$. 
Substituting $u=\CC_{I_w^{\rev}}(u_0)\CC_{I_{w'}^{\rev}}(u_0)$ in \eqref{AG:charmap_eq} we obtain the decomposition of the 
product $\bb_{I_w}\bb_{I_{w'}}$ on the basis $(\ab_{I_w})$, and we rewrite it in terms of the $\bb_{I_w}$ by using the transition matrix.
This also explains how to decompose any $\bb_I$ ($I$ not necessarily reduced) on a given basis of $\bb_{I_w}$. Again, it suffices to substitute $u=\CC_{I}(u_0)$ in \eqref{AG:charmap_eq} and use the transition matrix.

\begin{rem} \label{AM:algogood_rem}
Note that this algorithm is an improvement over the one in \cite{Bressler92} or \cite{Hornbostel09_pre} in the following sense. In both of these articles, it is explained how to decompose a product of two generators $\bb_I$ and $\bb_J$ as a linear combination of other such generators. 
But starting from a basis (\ie the $\bb_{I_w}$ for a choice of a reduced decomposition for each $w \in W$), it is not explained how to obtain a linear combination containing only generators $\bb_I$ with $I$ among the $I_w$ and an algorithm for redecomposing any given $\bb_I$ on a chosen basis is not given either. 
This is crucial to compute multiplication tables. 
\end{rem}


\section{Landweber-Novikov operations} \label{LN}

In this section we provide an algorithm for computing 
the Landweber-Novikov operations $S^{LN}$ on $\Omega^*(G/B)$.

\begin{ntt}
Let us recall briefly the definition of $S^{LN}$ 
(details can
be found in \cite[\S~4.1.9]{Levine07}). Consider a graded
polynomial ring $\ZZ[\TT]=\ZZ[t_1,t_2,\ldots,t_k,\ldots]$
in infinite number of variables; for a multi-index $I=(i_1,i_2,\ldots,i_k)$
we set
$$
{\TT}^I=t_1^{i_1}t_2^{i_2}\ldots t_k^{i_k}.
$$
Let $\lambda_{(\TT)}$ denote the formal
power series
$$
\lambda_{(\TT)}(x)=x+\sum_{i=1}^\infty t_ix^{i+1}.
$$
Consider a twisted theory $\tilde{\Omega}$ of 
$\Omega[\TT]=\Omega\otimes_{\ZZ}\ZZ[\TT]$ (see \cite[\S4]{Merkurjev02}).
By definition $\tilde\Omega(X)=\Omega(X)[\TT]$ for any $X$,
its Chern class is given by the formula
\begin{equation}\label{LN:twisted_chern}
c_1^{\tilde{\Omega}}(\Lb)=\lambda_{(\TT)}(c_1^{\Omega}(\Lb)),
\end{equation}
and its formal group law is given by
$$
F(x,y)=\lambda_{(\TT)}(U(\lambda_{(\TT)}^{-1}(x),\lambda_{(\TT)}^{-1}(y))).
$$
By the universality of $\Omega$ there is a natural transformation
$\Omega\to\tilde\Omega$ given by
\begin{align*}
a&\mapsto\sum_I S^{LN}_{I}(a)t^I,\qquad a\in \Omega(X)
\end{align*}
where the components $S^{LN}_I$ are called 
\emph{Landweber-Novikov operations} on $\Omega(X)$.
\end{ntt}

\begin{ntt}
By functoriality of $\RMF$ in $R$ (see \ref{FGR:funct})
the map $\LL\simeq\Omega(\pt)\to\tilde\Omega(\pt)\simeq\LL[\TT]$
induces a homomorphism
$$
\FGR{\LL}{M}{U}\to\FGR{\LL[\TT]}{M}{F},
$$
while functoriality with respect to formal group laws
induces a homomorphism
\begin{align*}
\FGR{\LL[\TT]}{M}{F}&\to\FGR{\LL[\TT]}{M}{U_{\Omega[\TT]}}\\
x_{\mu}&\mapsto\lambda_{(\TT)}(x_\mu).
\end{align*}

By \eqref{LN:twisted_chern} 
we have the following commutative diagram:
$$
\xymatrix{
\FGR{\LL}{M}{U}\ar[r]\ar[d]_-{\cc^{\Omega}}&
\FGR{\LL[\TT]}{M}{F}
\ar[r]\ar[d]_-{\cc^{\tilde\Omega}}
&\FGR{\LL[\TT]}{M}{U_{\Omega[\TT]}}\ar[d]_-{\cc^{\Omega[\TT]}}\\
\Omega(G/B)\ar[r]&\tilde\Omega(G/B)\ar[r]&\Omega(G/B)[\TT].
}
$$
\end{ntt}

\begin{ntt}
An action of the Landweber-Novikov operation $S^{LN}_I$
on a basis element $\bb_{I_w}$ can be
computed as follows:
First, we compute the image of $u=\CC_{I_w^{\rev}}(u_0)$ under
the composition of top horizontal arrows. Second, we extract  
the coefficient at the monomial ${\TT^I}$ of this image.
Finally, we apply the characteristic map $\cc^\Omega$
to that coefficient. The result will give
$S^{LN}_I(\bb_{I_w})$.

Indeed, by definition
$S_I^{LN}(\bb_{I_w})$ is equal to the coefficient at ${\TT^I}$
of the image of $\bb_{I_w}$ under the composition
of bottom horizontal arrows. Since $\tor\; \bb_{I}=\cc_{G/B}\big(\CC_{I^{\rev}}(u_0)\big)$
and by commutativity of the diagram we are done.
\end{ntt}

\section{Examples of computations} \label{EG}
In the present section 
we list the multiplication tables for rings $\Omega^*(G/B)$,
where $G$ has rank $2$. The results are obtained by means of the algorithm
described in Section~\ref{AM} and realized in Macaulay~2 packages \cite{M2package}.
The answers for the other oriented cohomology theories are easily derived by
a specialization of the coefficients of the universal formal group law. For instance, the answer for connective $K$-theory is obtained by specializing the coefficient $a_1$ to $v$ and all others to zero.

\begin{ntt}
We use the presentation of the Lazard ring
$$
\LL=\ZZ[a_1,a_2,\ldots],
$$
where the first generators $a_i$ are the following linear combinations of the coefficients $a_{ij}$ of the universal formal group law $U$:
\begin{align*}
&a_1=a_{11};\\
&a_2=a_{12};\\
&a_3=a_{22}-a_{13};\\
&a_4=a_{14};\\
&a_5=-9a_{15}+a_{24}+2a_{33}.
\end{align*}
Note that the $a_i$ (resp.~the $a_{ij}$) are of cohomological degree $-i$ (resp.~$a_{1-i-j}$). For root systems of rank $2$, the longest element is of length at most $N=6$ (in the $G_2$ case), and we therefore need to compute the universal formal group law up to order $7$, which thus only involves $a_1, \ldots, a_6$. In fact, $a_6$ does not appear in the formulas: it is not difficult to show that $a_N$ will not appear in the multiplication formulas for a root system with longest element of length $N$. 
\end{ntt}

\begin{ntt}
We use the upper case letter $\Zb_I$ for the element $\bb_I=q_{I*}(1)$. For brevity $\Zb_I$ with $I=(i_1,\ldots,i_l)$ is denoted just by $\Zb_{i_1\ldots i_l}$, and when $I$ is the empty sequence, $\Zb_I$ is denoted by $\pt$.
Note that when
$l(w)+l(w')\ge N=\dim G/B$ one has 
$$\Zb_{I_w}\Zb_{I_{w'}}=\delta_{w,w_0w'}\pt$$
so we list the remaining cases only. 
\end{ntt}

\begin{ntt}[$A_2$ case]
\begin{align*}
\Zb_{121} & =1+a_2\Zb_1;\\
\Zb_{12}^2 & =\Zb_2;\\
\Zb_{21}^2 & =\Zb_1;\\
\Zb_{12}\Zb_{21} & =\Zb_1+\Zb_2+a_1\pt.
\end{align*}
This agrees with the computations of Hornbostel-Kiritchenko in \cite{Hornbostel09_pre}.
\end{ntt}

\begin{ntt}[$B_2$ case]
\begin{align*}
\Zb_{1212} & =1+2a_2\Zb_{12}+(a_3-a_1a_2)\Zb_2;\\
\Zb_{121}^2 & =\Zb_{21};\\
\Zb_{212}^2 & =2\Zb_{12}+a_1\Zb_2;\\
\Zb_{121}\Zb_{212} & =\Zb_{12}+\Zb_{21}+a_1\Zb_1+a_1\Zb_2+(2a_2+a_1^2)\pt;\\
\Zb_{121}\Zb_{12} & =\Zb_1+\Zb_2+a_1\pt;\\
\Zb_{121}\Zb_{21} & =\Zb_1;\\
\Zb_{212}\Zb_{12} & =\Zb_2;\\
\Zb_{212}\Zb_{21} & =2\Zb_1+\Zb_2+2a_1\pt.
\end{align*}
\end{ntt}

\begin{ntt}[$G_2$ case]
\begin{align*}
\Zb_{121212} & =1+4a_2\Zb_{1212}+(10a_3-10a_1a_2)\Zb_{212} \\
& \quad -(4a_4+9a_1a_3+3a_2^2-9a_1^2a_2)\Zb_{12}\\
& \quad -(54a_5-459a_1a_4-1188a_2a_3-108a_1^2a_3+1080a_1a_2^2+108a_1^3a_2)\Zb_2;\\
\Zb_{12121}^2 & =3\Zb_{2121}+3a_1\Zb_{121}+(13a_2+2a_1^2)\Zb_{21}+(2a_3+7a_1a_2+a_1^3)\Zb_1;\\
\Zb_{21212}^2 & =\Zb_{1212}+5a_2\Zb_{12}+(6a_3-5a_1a_2)\Zb_2;\\
\Zb_{12121}\Zb_{21212} & =\Zb_{1212}+\Zb_{2121}+a_1\Zb_{121}+a_1\Zb_{212}+(8a_2+a_1^2)\Zb_{12}+(8a_2+a_1^2)\Zb_{21}\\
& \quad +(4a_3+8a_1a_2+a_1^3)\Zb_1+(10a_3+6a_1a_2+a_1^3)\Zb_2 \\
& \quad +(-4a_4+a_1a_3+13a_2^2+15a_1^2a_2+a_1^4)\pt;\\
\Zb_{12121}\Zb_{1212} & =\Zb_{121}+3\Zb_{212}+4a_1\Zb_{12}+3a_1\Zb_{21}+(8a_2+4a_1^2)\Zb_1+(13a_2+5a_1^2)\Zb_2\\
& \quad +(a_3+16a_1a_2+5a_1^3)\pt;\\
\Zb_{12121}\Zb_{2121} & =2\Zb_{121}+2a_1\Zb_{21}+(4a_2+a_1^2)\Zb_1;\\
\Zb_{21212}\Zb_{1212} & =2\Zb_{212}+a_1\Zb_{12}+4a_2\Zb_2;\\
\Zb_{21212}\Zb_{2121} & =\Zb_{121}+\Zb_{212}+a_1\Zb_{12}+a_1\Zb_{21}+(5a_2+a_1^2)\Zb_1+(8a_2+a_1^2)\Zb_2\\
& \quad +(3a_3+6a_1a_2+a_1^3)\pt;\\
\Zb_{12121}\Zb_{121} & =3\Zb_{21}+2a_1\Zb_1;\\
\Zb_{12121}\Zb_{212} & =2\Zb_{12}+\Zb_{21}+2a_1\Zb_1+3a_1\Zb_2+(4a_2+3a_1^2)\pt;\\
\Zb_{21212}\Zb_{121} & =\Zb_{12}+2\Zb_{21}+2a_1\Zb_1+2a_1\Zb_2+(4a_2+2a_1^2)\pt;\\
\Zb_{21212}\Zb_{212} & =\Zb_{12};\\
\Zb_{1212}^2 & =2\Zb_{12}+a_1\Zb_2;\\
\Zb_{2121}^2 & =2\Zb_{21}+a_1\Zb_1;\\
\Zb_{1212}\Zb_{2121} & =2\Zb_{12}+2\Zb_{21}+3a_1\Zb_1+4a_1\Zb_2+(4a_2+4a_1^2)\pt;\\
\Zb_{12121}\Zb_{12} & =\Zb_1+3\Zb_2+3a_1\pt;\\
\Zb_{12121}\Zb_{21} & =\Zb_1;\\
\Zb_{21212}\Zb_{12} & =\Zb_2;\\
\Zb_{21212}\Zb_{21} & =\Zb_1+\Zb_2+a_1\pt;\\
\Zb_{1212}\Zb_{121} & =2\Zb_1+3\Zb_2+4a_1\pt;\\
\Zb_{1212}\Zb_{212} & =\Zb_2;\\
\Zb_{2121}\Zb_{121} & =\Zb_1;\\
\Zb_{2121}\Zb_{212} & =\Zb_1+2\Zb_2+2a_1\pt.
\end{align*}
\end{ntt}

\bibliographystyle{plain}

\end{document}